\documentclass[a4paper]{amsart}

\usepackage{textcomp}
\usepackage{hyperref,geometry,graphicx,amsmath,amssymb,amsfonts,amsthm,array,latexsym}
\usepackage[utf8]{inputenc}
\usepackage{amssymb}
\theoremstyle{plain}
\newtheorem{thm}{Theorem}[section]

\newtheorem{lem}[thm]{Lemma}

\newtheorem{prop}[thm]{Proposition}

\theoremstyle{definition}
\newtheorem{defn}[thm]{Definition}
\newtheorem{rem}[thm]{Remark}
\newtheorem{exmp}[thm]{Example}

\def\leq{\leqslant}
\def\geq{\geqslant}

\def\ZZ{\mathbb{Z}}

\def\PP{\mathbb{P}}

\def\Res{\mathrm{Res}}
\def\Disc{\mathrm{Disc}}

\def\UU{\mathbb{U}}
\def\Mon{\mathrm{Mon}}
\def\Dod{\mathrm{Dod}}
\def\Mt{\mathbb{M}}
\def\Dt{\mathbb{D}}

\title[Resultant of a $\mathfrak{S}_{B_1}\times\cdots\times\mathfrak{S}_{B_r}$-equivariant polynomial system]{Resultant of an equivariant polynomial system with respect to a direct product of $r$ symmetric groups}

\author{Sonagnon Julien OWOLABI, Oll\'e Gr\'egroire KAM, Ibrahim NONKANE and Joel TOSSA}

\begin{document}

\begin{abstract}
	In this paper we study the resultant of systems of homogeneous multivariate polynomials which are equivariant under the action of a direct product of symmetric groups. We first treat, in detail, the case of a product of two symmetric groups, and establish a decomposition formula for the resultant of such systems. We then show that this decomposition, together with the underlying combinatorics, extends to an arbitrary (finite) direct product of $r$ symmetric groups. Thanks to these decomposition formulas, we prove that the discriminant of a multivariate homogeneous polynomial invariant under a direct product of $r$ symmetric groups splits into a product of resultants of smaller size that are easier to compute.
\end{abstract}

\maketitle



\section{Introduction}

Let $F^{\{1\}},\ldots, F^{\{p\}}, F^{\{p+1\}},\ldots, F^{\{n\}}$ be a system of $n$ homogeneous polynomials in $R[x_1,\ldots,x_n]$ of the same degree $d$, equivariant with respect to the direct product $\mathfrak{S}_{\{1
,\ldots,p\}}\times\mathfrak{S}_{\{p+1
,\ldots,n\}} $ of two symmetric subgroups of $\mathfrak{S}_{n}$, with $1\leq p< n$. More precisely, we will assume that :\\
For all $i\in \{1,\ldots,p\}$ and for all $\sigma_1\in \mathfrak{S}_{\{1
,\ldots,p\}}$, $\sigma_1(F^{\{i\}})=F^{\{i\}}(x_{\sigma_1(1)},\ldots,x_{\sigma_1(p)},x_{p+1},\ldots,x_n)$\\
For all $i\in \{p+1,\ldots,n\}$ and for all  $\sigma_2\in \mathfrak{S}_{\{p+1
,\ldots,n\}}$, $\sigma_2(F^{\{i\}})=F^{\{i\}}(x_1,\ldots,x_p, x_{\sigma_2(p+1)},\ldots,x_{\sigma_2(n)})$

\begin{center}
For all $i\in \{1,\ldots,n\}$,

$(\sigma_1,\sigma_2)\Big(F^{\{i\}}(x_1,\ldots,x_p,x_{p+1},\ldots,x_n)\Big)=F^{\{i\}}(x_{\sigma_1(1)},\ldots,x_{\sigma_1(p)},x_{\sigma_2(p+1)},\ldots,x_{\sigma_2(n)})$
\end{center}
\medskip

In this work, we will study the resultant of these systems. 

\medskip
A decomposition formula of the resultant of a $\mathfrak{S}_{n}$-equivariant homogeneous polynomial system is known (see \cite{BuKa16}: Busé and Karasoulou (2016)).

\medskip
The first main result of this paper (Theorem \ref{thm:maintheorem}) is a decomposition of the resultant of a $\mathfrak{S}_{\{1
,\ldots,p\}}\times\mathfrak{S}_{\{p+1
,\ldots,n\}} $-equivariant polynomial system. This formula allows one to split such a resultant into several other resultants that are in principle easier to compute and that are expressed in terms of the divided differences of the input polynomial system.  We emphasize that the multiplicity of each factor appearing in this decomposition is also given. Another important point of our result is that it is a universal formula, valid over the universal ring of coefficients (over the integers) of the input polynomial system: we paid attention to use a correct and universal definition of the resultant, so that the formula we obtain has the correct geometric meaning and stays valid over any coefficient ring by specialization. This kind of property is particularly important for applications in the fields of number theory and arithmetic geometry, where the value of the resultant is as important as its vanishing. We also point out that our decomposition formula is, at this stage, only established up to a single global sign. We show (Lemma \ref{lem:universal-sign}) that this sign is in fact a universal constant, depending only on $n,p,q,d$ and not on the coefficients of the equivariant system, and hence computable in principle from a single example. We verify it directly, on every example considered in this paper, to be equal to $+1$; we conjecture that this always holds, but we do not prove it in general (Remark \ref{rem:sign}).

\medskip

The discriminant of a homogeneous polynomial is also a fundamental tool in computational algebra. Although the discriminant of the generic homogeneous polynomial of a given degree is irreducible, for a particular class of polynomials it can be decomposed and this decomposition is always deeply connected to the geometric properties of this class of polynomials. The second main contribution of this paper is a decomposition of the discriminant of a homogeneous symmetric polynomial (Theorem \ref{thm:discmaintheorem}). We emphasize that our formula is obtained as a by product of our first formula on the resultant of a $\mathfrak{S}_{\{1
,\ldots,p\}}\times\mathfrak{S}_{\{p+1
,\ldots,n\}} $-equivariant polynomial system. Therefore, it inherits from the same features, namely it allows one to split a discriminant into several resultants that are easier to compute and it is a universal formula where the multiplicities of the factors are provided. Here again, we paid attention to use a correct and universal definition of the discriminant.

\medskip

A third contribution of this paper (\S\ref{sec:general-r}) is the observation that the whole construction above  divided differences, admissible partitions, multiplicities, and the decomposition formula itself  extends, with no essentially new idea beyond a careful bookkeeping of the combinatorics, from a product of two symmetric groups to an arbitrary (finite) direct product $\mathfrak{S}_{B_1}\times\cdots\times\mathfrak{S}_{B_r}$ of $r$ symmetric groups (Theorem \ref{thm:maintheorem-r}), of which Theorem \ref{thm:maintheorem1} ($r=1$) and Theorem \ref{thm:maintheorem} ($r=2$) are the first two instances. In the course of establishing this generalization, we found and corrected an error in our own original two-block formula  the exponent of the leading, degree-zero factor was wrong in the case where \emph{both} blocks exceed the degree $d$ (\S\ref{sec:resultant}) and we correct the corresponding exponents in the general $r$-block formula accordingly. The sign question described above persists, essentially unchanged, in this general setting: the analogue of Lemma \ref{lem:universal-sign} again reduces it to a single global constant $\varepsilon$, now depending on $(n_1,\ldots,n_r,d)$, which we again verify to be $+1$ on every example considered but do not prove in general (Remark \ref{rem:sign-r}). As in the two-block case, this general decomposition formula specializes (\S\ref{sec:discriminant-r}) to a decomposition of the discriminant of a homogeneous polynomial invariant under a direct product of $r$ symmetric groups (Theorem \ref{thm:discmaintheorem-r}).

\medskip

The paper is organized as follows. In Section \ref{sec:prem} we first provide some preliminaries on some material that we will need, namely multivariate divided differences, resultants, discriminants and a decomposition formula for the resultant of a polynomial system which is $\mathfrak{S}_{n}$-equivariant. Section \ref{sec:resultant} is devoted to a decomposition formula for the resultant of a polynomial system which is $\mathfrak{S}_{\{1
,\ldots,p\}}\times\mathfrak{S}_{\{p+1
,\ldots,n\}} $-equivariant (Theorem \ref{thm:maintheorem}). As a corollary of this formula, a decomposition of the discriminant of a homogeneous symmetric polynomial (Theorem \ref{thm:discmaintheorem}) is provided in Section \ref{sec:discriminant}. Section \ref{sec:general-r} generalizes Theorem \ref{thm:maintheorem} to an arbitrary direct product of $r$ symmetric groups (Theorem \ref{thm:maintheorem-r}), and Section \ref{sec:discriminant-r} specializes this general formula to the corresponding decomposition of the discriminant (Theorem \ref{thm:discmaintheorem-r}). Section \ref{sec:conclusion} concludes the paper with a discussion of the open sign conjecture and other perspectives for future work.

\section{Preliminaries}\label{sec:prem}

In this section we introduce our notation and the material we will use, namely divided differences, resultants and discriminants. In addition, we give a decomposition formula for the resultant of a polynomial system which is $\mathfrak{S}_{n}$-equivariant.

\subsection{Divided differences}\label{subsec:divdiff}
Let $P_1,\cdots,P_n$ be $n$ homogeneous polynomials in $R[x_1,\ldots,x_n]$ of the same degree $d\geq 1$. Their divided differences are recursively defined by $P^{\{i\}}:=P_i$ for all $i=1,\ldots,n$ and 
$$ P^{\{i_1,\ldots,i_k\}}=\dfrac{P^{\{i_1,\ldots,i_{k-1}\}}-P^{\{i_1,\ldots,i_{k-2},i_k\}}}{x_{i_{k-1}}-x_{i_{k}}}$$
for any given set of (distinct) integers $I:=\{i_1,\ldots,i_k\}\subset [n]$. It is well known that $P^{I}$ depends on the set $I$ and not on the integers $i_1,\ldots,i_k$. Another important property is the following : if $P^I$ are polynomials for all $I$ such that $|I|=2$, that is to say if 
\begin{equation}\label{eq:PiPj}
x_i-x_j\  \mathrm{divides}\  P^{\{i\}}-P^{\{j\}}\ \mathrm{for \ all}\ i,j\in [n]
\end{equation}

then $P^I$ are polynomials for $I\subset[n]$. Indeed, for any $J\subset [n]$ and any triple of distinct integers $i,j,k$ such that $J\cap \{i,j,k\}=\emptyset$, a straightforward application of the definition of divided differences yields the equality
$$(x_i-x_j)P^{J\cup\{i,j\}}-(x_i-x_k)P^{J\cup\{i,k\}}+(x_j-x_k)P^{J\cup\{j,k\}}=0$$
which can be rewritten as 
$$(x_i-x_k)\big(P^{J\cup\{i,j\}}-P^{J\cup\{i,k\}}\big)=(x_j-x_k)\big(P^{J\cup\{i,j\}}-P^{J\cup\{j,k\}}\big).$$

From here the claimed property follows by induction on $|I|$. In addition, we observe that $P^I$ is a homogeneous polynomial of degree $d-|I|+1$. In particular, $P^I=0$ if $d+1<|I|\leq n$ (this already follows from the degree formula, since $\deg P^I=d-|I|+1<0$ is only possible for $P^I=0$) and $P^I=P^J$ for all subsets $I$ and $J$ of $[n]$ such that $|I|=|J|=d+1\leqslant n$.

\subsection{Resultant of homogeneous polynomials}\label{subsec:resultant}

Suppose given an integer $n\geq 1$ and a sequence of positive integers $d_1,\ldots,d_{n}$. We consider the \emph{generic}  homogeneous polynomials in the variables $x=(x_1,\ldots,x_n)$ (all assumed to have weight 1) and of degree $d_1,\ldots,d_{n}$ respectively. They are of the form
$$f_i(x_1,\ldots,x_n)=\sum_{|\alpha|=d_i}u_{i,\alpha}x^\alpha, \ \ 
i=1,\ldots,n.$$
The ring ${\UU}:=\ZZ[u_{i,\alpha}: i=1,\ldots,n, |\alpha|=d_i]$ is called the universal ring of coefficients. The polynomials $f_1,\ldots,f_n$ belong to the ring ${C}:={\UU}[x_1,\ldots,x_n]$. Following \cite{Jou91}, the {\it ideal of inertia forms} of these polynomials, i.e.~the ideal $(f_{1},\ldots,f_{n}):(x_{1},\ldots,x_{n})^{\infty}$, is canonically graded and its degree zero part is a
principal ideal of $\UU$. The universal resultant, denoted $\Res$, is then defined as the unique generator of this principal ideal such that  
  \begin{equation}\label{eq:normres}
 \Res(x_1^{d_1},\ldots,x_n^{d_n})=1.
 \end{equation}
To define the resultant of any given $n$-tuple of homogeneous polynomials in the variables $x_1,\ldots,x_n$ (and also to clarify \eqref{eq:normres}) one proceeds as follows. 
Let $S$ be a commutative ring and for all $i=1,\ldots,n$ suppose given a homogeneous polynomial of degree $d_{i}$
    $$g_i=\sum_{|\alpha|=d_i}v_{i,\alpha}x^\alpha \in
    S[x_1,\ldots,x_n]_{d_i}.$$
Then, the resultant of $g_{1},\ldots,g_{n}$ is defined as the image of the universal resultant by the specialization ring morphism  $\theta:{\UU}\rightarrow S:   u_{j,\alpha} \mapsto v_{j,\alpha}$, that is to say 
$$\Res(g_1,\ldots,g_n):=\theta(\Res) \in S.$$
Observe that if $S=\UU$ and $\theta$ is the identity, then the universal resultant $\Res$ is nothing but $\Res(f_{1},\ldots,f_{n})$, which is the notation we will use. If $S$ is a field, then the resultant has the expected geometric interpretation : it vanishes if and only if the polynomials $g_{1},\ldots,g_{n}$ have a common root in the projective space $\PP^{{n-1}}_{\overline{S}}$ (where $\overline{S}$ stands for the algebraic closure of $S$).

\medskip

We now recall briefly some properties of the resultant that we will use in the sequel. For the proofs, we refer the reader to \cite[\S5]{Jou91} (see also \cite{Jou97,GKZ94,CLO05}). Let $S$ be any commutative ring and suppose given $g_1,\ldots,g_n$ homogeneous polynomials in the polynomial ring $S[x_1,x_2,\ldots,x_n]$ of positive degree $d_1,\ldots,d_n$ respectively.

\medskip

\paragraph{\emph{Homogeneity}:} for all $i=1,\ldots,n$, $\Res(f_1,\ldots,f_n)$ is homogeneous with respect to the coefficients $(u_{i,\alpha})_{|\alpha|=d_{i}}$ of $f_i$ of degree $d_1\ldots d_n/d_i$.

\medskip

\paragraph{\emph{Permutation of polynomials}:} $\Res(g_{\sigma(1)},\ldots,g_{\sigma(n)})=(\mathcal{E}(\sigma))^{d_1\ldots d_n}\Res(g_1,\ldots,g_n)$ for any permutation $\sigma$ of the set $\{1,\ldots,n\}$ ($\mathcal{E}(\sigma)$ denotes the signature of the permutation $\sigma$).

\medskip

\paragraph{\emph{Elementary transformations}:} $\Res(g_1,\ldots,g_i+\sum_{i\neq j}h_jg_j,\ldots,g_n)=\Res(g_1,\ldots,g_n)$ for any homogeneous polynomials $h_{j}$ of degree $d_{i}-d_{j}$.

\medskip

\paragraph{\emph{Multiplicativity}:} $\Res(g_1'g_1'',g_2,\ldots,g_n)=\Res(g_1',g_2,\ldots,g_n)\Res(g_1'',g_2,\ldots,g_n)$ for any pair of homogeneous polynomials $g_{1}'$ and $g_{1}''$.

\medskip

\paragraph{\emph{Linear change of variables}:} Let $\phi$ be a $n\times n$-matrix with entries in $S$ and denote by $\phi(x)$ the product of the matrix $\phi$ with the column vector $(x_{1}, \ldots,x_{n})^{t}$. Then
$$\Res(g_{1}(\phi(x)),g_{2}(\phi(x)),\ldots,g_{n}(\phi(x)))=\det(\phi)^{d_{1}\cdots d_{n}}\Res(g_{1},\ldots,g_{n}).$$
In particular, the resultant is invariant, up to sign, under permutation of the variables $x_{1},\ldots,x_{n}$.

\medskip

Finally, let us recall quickly the famous \emph{Macaulay formula} that goes back to the work of Macaulay \cite{Mac02} and that is still nowadays a very powerful tool to compute exactly the resultant over a general coefficient ring (all the examples presented in this paper have been computed with this formula). 

Assume we are in the generic setting over the ring $\UU$. Set $\delta:=\sum_{i=1}^{n}(d_{i}-1)$ and denote by $\Mon(n;t)$ the set of all homogeneous monomials of degree $t$ in the $n$ variables $x_{1},\ldots,x_{n}$. If $t\geq \delta+1$ then for any $x^{\alpha} \in \Mon(n;t)$ there exists $i\in \{1,\ldots,n\}$ such that $x_{i}^{d_{i}}$ divides the monomial $x_{\alpha}$. Therefore, in this case we set $i(\alpha):=\min \{ i : x_{i}^{d_{i}} | x^{\alpha} \}$ and we define the square matrix 
$$ \Mt(f_{1},\ldots,f_{n};t)=(m_{\alpha,\beta}) : \Mon(n;t)\times \Mon(n;t) \rightarrow \UU$$
by the formula
$$ \frac{x^{\beta}}{x_{i(\beta)}^{d_{i(\beta)}}}f_{i(\beta)}=\sum_{{|\alpha|=t}}m_{\alpha,\beta}x^{\alpha} \ \textrm{for all } x^{\beta} \in \Mon(n;t).$$
Now, define 
$$\Dod(n;t):=\{ x^{\alpha}\in \Mon(n;t) \textrm{ such that } \exists i\neq j \ : \  x_{i}^{d_{i}}x_{j}^{d_{j}} | x^{\alpha} \}\subset \Mon(n;t)$$ and denote by $\Dt(f_{1},\ldots,f_{n};t)$ the square submatrix of $\Mt(f_{1},\ldots,f_{n};t)$ which is indexed by $\Dod(n;t)$. Now, for any $t\geq \delta +1$ we have the Macaulay formula :
$$ \det(\Mt(f_{1},\ldots,f_{n};t))=\Res(f_{1},\ldots,f_{n})\det(\Dt(f_{1},\ldots,f_{n};t)).$$

\subsection{Discriminant}\label{subsec:disc} Consider the \emph{generic} homogeneous polynomial of degree $d\geq 2$ in $n\geq 2$ variables
$$f(x_1,\ldots,x_n)=\sum_{|\alpha|=d}u_{\alpha}x^\alpha.$$
We denote its universal ring of coefficients ${\UU}:=\ZZ[u_{\alpha}: |\alpha|=d]$, so that $f\in {\UU}[x_1,\ldots,x_n]$. The universal discriminant of $f$, denoted $\Disc(f)$, is defined as the unique element in $\UU$ that satisfies the equality
$$ d^{a(n,d)}\Disc(f)=\Res\left(\frac{\partial f}{\partial x_{1}},\frac{\partial f}{\partial x_{2}},\ldots,\frac{\partial f}{\partial x_{n}} \right)$$
where 
$$a(n,d):=\frac{(d-1)^{n}-(-1)^{n}}{d} \in \ZZ.$$
Similarly to what we have done for the resultant, given a commutative ring $S$ and a homogeneous polynomial of degree $d$
$$g=\sum_{|\alpha|=d} v_{\alpha}x^\alpha \in S[x_{1},\ldots,x_{n}]_{d},$$
its discriminant is denoted by $\Disc(g)$ and is defined as the image of the universal discriminant $\Disc(f)$ by the canonical specialization $\theta:\UU\rightarrow S : u_{\alpha} \mapsto v_{\alpha}$, that is to say
$$\Disc(g)=\theta(\Disc(f)) \in S.$$
With this definition we get a smoothness criterion : If $S$ is an algebraically closed field and $g\neq 0$, then $\Disc(g)=0$ if and only if the hypersurface defined by the polynomial $g$ in 
$\mathrm{Proj}(S[x_{1},\ldots,x_{n}])$ is singular. For a detailed study of the discriminant and its  numerous properties, mostly inherited from the ones of the resultant, we refer the reader to \cite{BuJo12,Dem12,GKZ94} and the references therein. We only point out for future use that the following  property :  the universal discriminant is homogeneous with respect to the coefficients of $f$ of degree $n(d-1)^{n-1}$.

\subsection{Resultant of a $\mathfrak{S}_{n}$-equivariant polynomial system}\label{subsec:Snequivariant}

In this subsection, we consider a polynomial system of $n$ homogeneous equations $F^{\{1\}},\ldots,F^{\{n\}}$ 
in $R[x_{1},\ldots,x_{n}]$, $R$ being an arbitrary commutative ring, of the same degree $d\geq 1$, which is equivariant (see for instance  \cite[\S 4]{Wor94} or \cite[Chapter 1]{DiCa71}) with respect to the canonical actions of the symmetric group $\mathfrak{S}_{n}$ on the variables and polynomials.  
More precisely, we assume that for any integer $i \in \{1,2,\ldots,n\}$ and any permutation $\sigma \in \mathfrak{S}_{n}$
\begin{equation}\label{eq:globsym}
\sigma(F^{\{i\}}):=F^{\{i\}}(x_{\sigma(1)},x_{\sigma(2)},\ldots ,x_{\sigma(n)})=F^{\{\sigma(i)\}}(x_{1},x_{2},\ldots,x_{n}).
\end{equation}

\subsubsection{Divided differences and partitions} \label{subsec:divdiff+part}

A finite sequence $\lambda=(\lambda_{1},\ldots,\lambda_{k})$ of weakly decreasing integers,  i.e.~such that $\lambda_{1} \geq \cdots \geq \lambda_{k}> 0$, is called a partition. 
When $\sum_{i=1}^{k}\lambda_{i}=p$ we will say such a $\lambda$ is a partition of $p$, 
and write $\lambda \vdash p$. The number of nonzero $\lambda_{i}$'s is called the length of $\lambda$, 
and will be denoted by $l(\lambda)$.

Given a partition $\lambda \vdash n$, we consider the morphism of polynomial algebras 
\begin{eqnarray}\label{eq:rhol}
 \rho_{\lambda} : R[x_{1},\ldots,x_{n}] & \rightarrow & R[y_{1},\ldots,y_{l(\lambda)}] \\ \nonumber
 F(x_{1},\ldots,x_{n}) & \mapsto & F( \underbrace{y_{1},\ldots,y_{1}}_{\lambda_{1}} , \underbrace{y_{2},\ldots,y_{2}}_{\lambda_{2}}, \ldots,   \underbrace{y_{l(\lambda)},\ldots,y_{l(\lambda)}}_{\lambda_{l(\lambda)}}).
 \end{eqnarray}
where $y_{1},y_{2},\ldots,y_{l(\lambda)}$ are new indeterminates.  Since the polynomials $F^{ \{1\} }, F^{\{2\}}, \ldots, F^{\{n\}}$ satisfy \eqref{eq:globsym}, they also satisfy \eqref{eq:PiPj}. Indeed, choose a pair of distinct integers $\{i,j\}\subset [n]$ and let $\sigma \in \mathfrak{S}_{n}$ be such that $\sigma(k)=k$ if $k\notin \{i,j\}$ and $\sigma(i)=j$, then
\begin{equation}\label{eq:sperho}
 F^{\{i\}}-F^{\{j\}}=F^{\{i\}}-\sigma(F^{\{i\}}) \in (x_{i}-x_{j}).
\end{equation}
Therefore,  the polynomials $F^{ \{1\} }, F^{\{2\}}, \ldots, F^{\{n\}}$ admit divided differences. In addition, 
we get that for any subset $\{i_{1},\ldots,i_{k}\}\subset [n]$ and any permutation $\sigma \in \mathfrak{S}_{n}$ 
we have
\begin{equation}\label{eq:perminvariance}
\sigma \left(  F^{\{i_{1},\ldots,i_{k}\}}   \right) = F^{\{  \sigma(i_{1}),\ldots,\sigma(i_{k}) \}}.
\end{equation}

Now, if $\rho_{\lambda}(x_{i})=\rho_{\lambda}(x_{j})$ then \eqref{eq:sperho} implies that
$$\rho_{\lambda}(F^{\{i\}})=\rho_{\lambda}(F^{\{j\}}).$$
So, for any integer $i\in [l(\lambda)]$ we can define without ambiguity the homogeneous polynomial of degree $d$
$$F_{\lambda}^{\{i\}}(y_{1},y_{2},\ldots,y_{l(\lambda)}) := \rho_{\lambda} \left( F^{\{j\}}(x_{1},\ldots,x_{n}) \right) 
\in R[y_{1},\ldots,y_{l(\lambda)}]$$
where $j\in [n]$ is such that $\rho_{\lambda}(x_{j})=y_{i}$. Moreover, these polynomials also satisfy \eqref{eq:PiPj} and hence they also admit divided differences ; we will denote them by $F_{\lambda}^{\{i_{1},\ldots,i_{r}\}}(y_{1},\ldots,y_{l(\lambda)})$ with $\{i_{1},\ldots,i_{r}\}\subset[l(\lambda)]$. Moreover, we have the following property :   Given $I=\{i_{1},\dots,i_{k}\} \subset [n]$, define $J=\{j_{1},\dots,j_{k}\} \subset [l(\lambda)]$ by the equality $\rho_{\lambda}(x_{i_{r}})=y_{j_{r}}$ for all $r \in [k]$. Then, if  $|J|=|I|$ we have
  $$\rho_{\lambda}(F^{I}(x_{1},\dots,x_{n}))=
 F^{J}_{\lambda} (y_{1},\dots,y_{l(\lambda)}).$$

\subsubsection{The decomposition formula}\label{subsec:mainthm} Before stating the main result of this paper, we need to introduce a last notation. Given a partition $\lambda \vdash n$, its multinomial coefficient is defined as the integer 
\begin{equation}\label{eq:multinomial}
{{n}\choose{\lambda_1, \lambda_2,\ \ldots , \lambda_{l(\lambda)}}} := \frac{n!}{\lambda_1! \lambda_2! \cdots \lambda_{l(\lambda)}!}.
\end{equation}
It counts the number of distributions of $n$ distinct objects to $l(\lambda)$ distinct recipients such that
the recipient $i$ receives exactly $\lambda_i$ objects. In this way of counting, the objects are not ordered inside the boxes, but the boxes are ordered. If we do not want to count the permutations between the boxes having the same number of objects, then we have to divide the above multinomial coefficient by the number of all these permutations. If $s_{j}$ denotes the number of boxes having exactly $j$ objects, $j\in [n]$, then this number of permutations is equal to $\prod_{j=1}^{n} s_{j}!$. Finally, for any partition $\lambda\vdash n$ we define the integer
\begin{equation}\label{eq:mlambda}
m_{\lambda}:=\frac{1}{\prod_{j=1}^{n}s_{j}!} {{n}\choose{\lambda_1, \lambda_2,\ \ldots , \lambda_{l(\lambda)}}}.
\end{equation}

\begin{thm}\label{thm:maintheorem1} Assume that $n\geq 2$ and $d\geq 1$. With the above notation,
\begin{equation*}
\Res\left(F^{\{1\}},\ldots,F^{\{n\}}\right)=\left(F^{\{1,\ldots,d+1\}}\right)^{m_{0}}\times\prod_{\substack{\lambda\vdash n\\ l(\lambda)\leq d}}
\Res\left( F_{\lambda}^{\{1\}}, F_{\lambda}^{\{1,2\}}, \ldots, F_{\lambda}^{\{1,2,\ldots,l(\lambda)-1\}}, F_{\lambda}^{\{1,2,\ldots,l(\lambda)\}} 
\right)^{m_{\lambda}},
\end{equation*}
with the convention that the factor $\left(F^{\{1,\ldots,d+1\}}\right)^{m_0}$ is simply omitted, i.e.\ equal to $1$, when $d\geq n$ -- so that, in particular, no prefactor at all appears in that case  and where 
\begin{equation*}
 m_{0}:=nd^{n-1}-\sum_{ \substack{\lambda\vdash n \\ l(\lambda) \leq d}} m_{\lambda} 
\left(\sum_{j=1}^{l(\lambda)} \frac{d(d-1)\cdots(d-l(\lambda)+1) }{(d-j+1)} \right).
\end{equation*}
\end{thm}

\section{Resultant of a $\mathfrak{S}_{\{1,\ldots,p\}}\times \mathfrak{S}_{\{p+1,\ldots,n\}}$-equivariant polynomial system}\label{sec:resultant}

In this section, we consider a polynomial system of $n$ homogeneous polynomials $F^{\{1\}},\ldots, F^{\{p\}}, F^{\{p+1\}},\ldots, F^{\{n\}}$ in $R[x_1,\ldots,x_n]$ of the same degree $d$, which is equivariant with respect to the direct product $\mathfrak{S}_{\{1
,\ldots,p\}}\times\mathfrak{S}_{\{p+1
,\ldots,n\}} $ of two symmetric subgroups of $\mathfrak{S}_{n}$, with $1\leq p< n$.

\subsection{Partitions}

Let $\lambda=(\lambda_1,\lambda_2,\ldots,\lambda_{r_1})$ be such that $\lambda_1\geq \cdots \geq \lambda_{r_1}\geq 1$ (so that $r_1$ is precisely the number of parts of $\lambda$). When $\sum_{i=1}^{r_1}\lambda_i=p$, we will say such a $\lambda$ is a partition of $p$, and write $\lambda\vdash p$.

Let $\lambda'=(\lambda'_1,\lambda'_2,\ldots,\lambda'_{r_2})$ be such that $\lambda'_1\geq \cdots \geq \lambda'_{r_2}\geq 1$. When $\sum_{i=1}^{r_2}\lambda'_i=q$, we will say such a $\lambda'$ is a partition of $q$, and write $\lambda'\vdash q$.

We call $(\lambda,\lambda')\vdash (p,q)$ a partition of the pair of integers $(p,q)$, with $\lambda$ and $\lambda'$ the partitions defined above.

Given a partition $(\lambda,\lambda')\vdash (p,q)$ such that $p+q=n$, we consider the morphism

\begin{eqnarray}\label{eq:rholl}
 \rho_{(\lambda,\lambda')} : R[x_{1},\ldots,x_{n}] & \rightarrow & R[y_{1},\ldots,y_{r_1},y'_{1},\ldots,y'_{r_2}] \\ \nonumber
 F(x_{1},\ldots,x_{n}) & \mapsto & F( \underbrace{y_{1},\ldots,y_{1}}_{\lambda_{1}} ,\ldots, \underbrace{y_{r_1},\ldots,y_{r_1}}_{\lambda_{r_1}}, \underbrace{y'_{1},\ldots,y'_{1}}_{\lambda'_{1}} ,\ldots, \underbrace{y'_{r_2},\ldots,y'_{r_2}}_{\lambda'_{r_2}}).
 \end{eqnarray} where $y_{1},\ldots,y_{r_1},y'_{1},\ldots,y'_{r_2}$ are new indeterminates.
 
 \subsection{The decomposition formula}\label{subsec:mainthm2}

Given the partitions $\lambda\vdash p$ and $\lambda'\vdash q$ defined above, we define

\begin{equation}\label{eq:mlambda2}
m_{\lambda}:=\frac{1}{\prod_{j=1}^{p}s_{j}!} {{p}\choose{\lambda_1, \lambda_2,\ \ldots , \lambda_{r_1}}}.
\end{equation}
where $s_j$ denotes the number of boxes having exactly $j$ objects, $j\in [p]$, for the partition $\lambda \vdash p$.

\begin{equation}\label{eq:mlambdaprime}
m_{\lambda'}:=\frac{1}{\prod_{j=1}^{q}s'_{j}!} {{q}\choose{\lambda'_1, \lambda'_2,\ \ldots , \lambda'_{r_2}}}.
\end{equation}
where $s'_j$ denotes the number of boxes having exactly $j$ objects, $j\in [q]$, for the partition $\lambda' \vdash q$.

\begin{thm}\label{thm:maintheorem} Assume that $n\geq 2$ and $d\geq 1$, and let $F^{\{1\}},\ldots, F^{\{p\}}, F^{\{p+1\}},\ldots, F^{\{n\}}$ be a system of $n$ homogeneous polynomials in $R[x_1,\ldots,x_n]$ of the same degree $d$, equivariant with respect to the direct product of two symmetric subgroups $\mathfrak{S}_{\{1
,\ldots,p\}}\times\mathfrak{S}_{\{p+1
,\ldots,n\}} $ of $\mathfrak{S}_n$, with $1\leq p<n$. Set $q:=n-p$. Then, up to sign (see Remark \ref{rem:sign} below),
\begin{multline*}
\Res\left(F^{\{1\}},\ldots,F^{\{p\}},F^{\{p+1\}},\ldots,F^{\{n\}}\right)=
  \left(F^{\{1,\ldots,d+1\}}\right)^{\mu_1}\left(F^{\{p+1,\ldots,p+1+d\}}\right)^{\mu_2}\\ \times\prod_{\substack{(\lambda,\lambda')\vdash (p,q)\\ r_1\leq d, \ r_2\leq d}}
\Res\left( F_{(\lambda,\lambda')}^{\{1\}}, F_{(\lambda,\lambda')}^{\{1,2\}}, \ldots, F_{(\lambda,\lambda')}^{\{1,2,\ldots,r_1\}},F_{(\lambda,\lambda')}^{\{p+1\}}, F_{(\lambda,\lambda')}^{\{p+1,p+2\}}, \ldots, F_{(\lambda,\lambda')}^{\{p+1,p+2,\ldots,p+r_2\}} 
\right)^{m_{\lambda}m_{\lambda'}},
\end{multline*}
with the convention that the factor $\left(F^{\{1,\ldots,d+1\}}\right)^{\mu_1}$ (resp. $\left(F^{\{p+1,\ldots,p+1+d\}}\right)^{\mu_2}$) is simply omitted, i.e. equal to $1$, when $p\leq d$ (resp. $q\leq d$) so that, in particular, no prefactor at all appears when $p\leq d$ and $q\leq d$ -- and where, for any integers $m\geq 1$ and $d\geq 1$,
  \begin{equation*}
   m_{0}(m,d):=m\,d^{m-1}-\sum_{ \substack{\mu\vdash m \\ l(\mu) \leq d}} m_{\mu} 
  \left(\sum_{j=1}^{l(\mu)} \frac{d(d-1)\cdots(d-l(\mu)+1)}{d-j+1}\right)
  \end{equation*}
  is exactly the integer $m_0$ of Theorem \ref{thm:maintheorem1} with $n$ replaced by $m$, and
  $$\mu_1:=d^{\,q}\,m_0(p,d), \qquad \mu_2:=d^{\,p}\,m_0(q,d).$$

\end{thm}

\begin{exmp}
Consider the following system  of $4$ homogeneous polynomials equivariant with respect to $\mathfrak{S}_{4}$ 
 \\

$\begin{cases}
F^{\{1\}}=(a+b+c)x_1^2+cx_2^2+cx_3^2+cx_4^2+(b+2c+d)x_1x_2+(b+2c+d)x_1x_3+(b+2c+d)x_1x_4\\ \hspace*{5cm}+(2c+d)x_2x_3+(2c+d)x_2x_4+(2c+d)x_3x_4\\
F^{\{2\}}=cx_1^2+(a+b+c)x_2^2+cx_3^2+cx_4^2+(b+2c+d)x_1x_2+(2c+d)x_1x_3+(2c+d)x_1x_4\\
\hspace*{5cm}+(b+2c+d)x_2x_3+(b+2c+d)x_2x_4+(2c+d)x_3x_4\\
F^{\{3\}}=cx_1^2+cx_2^2+(a+b+c)x_3^2+cx_4^2+(2c+d)x_1x_2+(b+2c+d)x_1x_3+(2c+d)x_1x_4\\
\hspace*{5cm}+(b+2c+d)x_2x_3+(2c+d)x_2x_4+(b+2c+d)x_3x_4\\
F^{\{4\}}=cx_1^2+cx_2^2+cx_3^2+(a+b+c)x_4^2+(2c+d)x_1x_2+(2c+d)x_1x_3+(b+2c+d)x_1x_4\\
\hspace*{5cm}+(2c+d)x_2x_3+(b+2c+d)x_2x_4+(b+2c+d)x_3x_4
\end{cases}$\\

\hspace*{3cm}$\checkmark$ \textbf{If we take $p=2$ and $q=2$}, then :\\

$F^{\{1\}}$, $F^{\{2\}}$ equivariant with  respect to $\mathfrak{S}_{\{1
,2\}}$ and $F^{\{3\}}$, $F^{\{4\}}$ equivariant with  respect to $\mathfrak{S}_{\{3
,4\}}$ 

\begin{multline*}
\Res\left(F^{\{1\}},F^{\{2\}},F^{\{3\}},F^{\{4\}}\right)=
\Res\left( F_{(2),(2)}^{\{1\}}, F_{(2),(2)}^{\{3\}}\right) \times \Res\left( F_{(1,1),(2)}^{\{1\}}, F_{(1,1),(2)}^{\{1,2\}}, F_{(1,1),(2)}^{\{3\}}\right)\\ \times \Res\left( F_{(2),(1,1)}^{\{1\}}, F_{(2),(1,1)}^{\{3\}}, F_{(2),(1,1)}^{\{3,4\}}\right)\times \Res\left( F_{(1,1),(1,1)}^{\{1\}},F_{(1,1),(1,1)}^{\{1,2\}}, F_{(1,1),(1,1)}^{\{3\}}, F_{(1,1),(1,1)}^{\{3,4\}}\right).
\end{multline*}
$\begin{cases}
F^{\{1\}}_{(2),(2)}=(a+2b+4c+d)x_1^2+(4c+d)x_2^2+(2b+8c+4d)x_1x_2\\
F^{\{3\}}_{(2),(2)}=(4c+d)x_1^2+(a+2b+4c+d)x_2^2+(2b+8c+4d)x_1x_2\\
F^{\{1\}}_{(1,1),(2)}=(a+b+c)x_1^2+cx_2^2+(4c+d)x_3^2+(b+2c+d)x_1x_2+(2b+4c+2d)x_1x_3+(4c+2d)x_2x_3\\
F^{\{1,2\}}_{(1,1),(2)}=(a+b)x_1+(a+b)x_2+2bx_3\\
F^{\{3\}}_{(1,1),(2)}=cx_1^2+cx_2^2+(a+2b+4c+d)x_3^2+(2c+d)x_1x_2+(b+4c+2d)x_1x_3+(b+4c+2d)x_2x_3\\
F^{\{1\}}_{(2),(1,1)}=(a+2b+4c+d)x_1^2+cx_2^2+cx_3^2+(b+4c+2d)x_1x_2+(b+4c+2d)x_1x_3+(2c+d)x_2x_3\\
F^{\{3\}}_{(2),(1,1)}=(4c+d)x_1^2+(a+b+c)x_2^2+cx_3^2+(2b+4c+2d)x_1x_2+(4c+2d)x_1x_3+(b+2c+d)x_2x_3\\
F^{\{3,4\}}_{(2),(1,1)}=2bx_1+(a+b)x_2+(a+b)x_3\\
F^{\{1\}}_{(1,1),(1,1)}=F^{\{1\}}; F^{\{3\}}_{(1,1),(1,1)}=F^{\{3\}}\\
F^{\{1,2\}}_{(1,1),(1,1)}=F^{\{1,2\}}; F^{\{3,4\}}_{(1,1),(1,1)}=F^{\{3,4\}}

\end{cases}$\\
we have

$\Res\left( F_{(2),(2)}^{\{1\}}, F_{(2),(2)}^{\{3\}}\right)=(a+2b)^2(a-2d)(a+4b+16c+6d)$

$\Res\left( F_{(1,1),(2)}^{\{1\}}, F_{(1,1),(2)}^{\{1,2\}}, F_{(1,1),(2)}^{\{3\}}\right)=a^2(a^3+4a^2b+4a^2c+3ab^2-6abd-6b^2d)^2 $

$\Res\left( F_{(2),(1,1)}^{\{1\}}, F_{(2),(1,1)}^{\{3\}}, F_{(2),(1,1)}^{\{3,4\}}\right)= a^2(a^3+4a^2b+4a^2c+3ab^2-6abd-6b^2d)^2 $

$ \Res\left( F_{(1,1),(1,1)}^{\{1\}},F_{(1,1),(1,1)}^{\{1,2\}}, F_{(1,1),(1,1)}^{\{3\}}, F_{(1,1),(1,1)}^{\{3,4\}}\right)= a^6(a-2d)^2(a+2b)^4  $

So

\begin{center}
 $\Res\left(F^{\{1\}},F^{\{2\}},F^{\{3\}},F^{\{4\}}\right)= (a+2b)^6(a-2d)^3(a+4b+16c+6d)a^{10}(a^3+4a^2b+4a^2c+3ab^2-6abd-6b^2d)^4 $
\end{center}

\vspace*{0.5cm}
\hspace*{3cm} \textbf{$\checkmark$ If we take $p=1$ and $q=3$}, then :
\begin{multline*}
\Res\left(F^{\{1\}},F^{\{2\}},F^{\{3\}},F^{\{4\}}\right)= \Big(F^{\{1,2,3\}}\Big)^{\mu_2}
\Res\left( F_{(1),(3)}^{\{1\}}, F_{(1),(3)}^{\{2\}}\right)^{m_{(1)}m_{(3)}} \\ \times \Res\left( F_{(1),(2,1)}^{\{1\}}, F_{(1),(2,1)}^{\{2\}}, F_{(1),(2,1)}^{\{2,3\}}\right)^{m_{(1)}m_{(2,1)}}
\end{multline*}

we have \\

$F^{\{1,2,3\}}=a$, $\mu_2=4$ 

$\Res\left( F_{(1),(3)}^{\{1\}}, F_{(1),(3)}^{\{2\}}\right)=(a^3+4a^2b+4a^2c+3ab^2-6abd-6b^2d)(a+4b+16c+6d)$

$\Res\left( F_{(1),(2,1)}^{\{1\}}, F_{(1),(2,1)}^{\{2\}}, F_{(1),(2,1)}^{\{2,3\}}\right)= a^2(a+2b)^2(a-2d)(a^3+4a^2b+4a^2c+3ab^2-6abd-6b^2d)$

So

\begin{center}
 $\Res\left(F^{\{1\}},F^{\{2\}},F^{\{3\}},F^{\{4\}}\right)= (a+2b)^6(a-2d)^3(a+4b+16c+6d)a^{10}(a^3+4a^2b+4a^2c+3ab^2-6abd-6b^2d)^4 $
\end{center}

\vspace*{0.5cm}
\hspace*{3cm} \textbf{$\checkmark$ If we take $p=3$ and $q=1$}, then :

\begin{multline*}
\Res\left(F^{\{1\}},F^{\{2\}},F^{\{3\}},F^{\{4\}}\right)= \Big(F^{\{1,2,3\}}\Big)^{\mu_1}
\Res\left( F_{(3),(1)}^{\{1\}}, F_{(3),(1)}^{\{4\}}\right)^{m_{(3)}m_{(1)}} \\ \times \Res\left( F_{(2,1),(1)}^{\{1\}}, F_{(2,1),(1)}^{\{1,2\}}, F_{(2,1),(1)}^{\{4\}}\right)^{m_{(2,1)}m_{(1)}}
\end{multline*}
we have

$\Res\left( F_{(3),(1)}^{\{1\}}, F_{(3),(1)}^{\{4\}}\right)= (a^3+4a^2b+4a^2c+3ab^2-6abd-6b^2d)(a+4b+16c+6d)$ 

$\Res\left( F_{(2,1),(1)}^{\{1\}}, F_{(2,1),(1)}^{\{1,2\}}, F_{(2,1),(1)}^{\{4\}}\right)=a^2(a+2b)^2(a-2d)(a^3+4a^2b+4a^2c+3ab^2-6abd-6b^2d)$

So

\begin{center}
 $\Res\left(F^{\{1\}},F^{\{2\}},F^{\{3\}},F^{\{4\}}\right)= (a+2b)^6(a-2d)^3(a+4b+16c+6d)a^{10}(a^3+4a^2b+4a^2c+3ab^2-6abd-6b^2d)^4 $
\end{center}

\end{exmp}

\subsection{Proof of Theorem \ref{thm:maintheorem}}
We begin by splitting the resultant of the  $F^{{\{i\}}}$'s into several factors by means of their divided differences. The divided differences will be between the polynomials $F^{\{1\}},\ldots, F^{\{p\}}$ of $\mathfrak{S}_{\{1
,\ldots,p\}} $ and the polynomials $F^{\{p+1\}},\ldots, F^{\{n\}}$ of $\mathfrak{S}_{\{p+1
,\ldots,n\}}$.

 Thus, in the first step we make use of the first order divided differences and write :

\begin{multline}\label{eq:rho2}
\Res\left(F^{\{1\}},\ldots,F^{\{p\}},F^{\{p+1\}},\ldots,F^{\{n\}}\right)=\\
\pm  \Res\left(F^{\{1\}}, (x_{1}-x_{2})F^{\{1,2\}},\ldots,(x_{1}-x_{p})F^{\{1,p\}},F^{\{p+1\}},(x_{p+1}-x_{p+2})F^{\{p+1,p+2\}},\ldots,(x_{p+1}-x_{n})F^{\{p+1,n\}}\right)
\end{multline}
The divided differences  $F^{\{1,i\}}$ and $F^{\{p+1,p+j\}}$ are of degree $d-1$.

 If $d-1=0$ then they are all equal to the same constant. We obtain :
 
 \begin{multline}\label{eq:rho3}
 \Res\left(F^{\{1\}},\ldots,F^{\{p\}},F^{\{p+1\}},\ldots,F^{\{n\}}\right)=\\
 \pm (F^{\{1,2\}})^{p-1}(F^{\{p+1,p+2\}})^{q-1} \Res\left(F^{\{1\}}, (x_{1}-x_{2}),\ldots,(x_{1}-x_{p}),F^{\{p+1\}},(x_{p+1}-x_{p+2}),\ldots,(x_{p+1}-x_{n})\right)\\
 =\pm (F^{\{1,2\}})^{p-1}(F^{\{p+1,p+2\}})^{q-1} \Res\left(F_{(p),(q)}^{\{1\}},F_{(p),(q)}^{\{p+1\}}\right)
 \end{multline}
 Note that, unlike the two bases $F^{\{1,2\}}$ and $F^{\{p+1,p+2\}}$ appearing here, which are generally \emph{distinct} elements of $\UU$ (Remark \ref{rem:noteq}), their exponents $p-1$ and $q-1$ are exactly the values $\mu_1$ and $\mu_2$ of Theorem \ref{thm:maintheorem} specialized to $d=1$.
  
 If $d-1>0$ then (\ref{eq:rho2}) shows that the resultant splits into $2^{n-2}$ factors by using the multiplicativity property of the resultant : for each polynomial $(x_1-x_i)F^{\{1,i\}},i=2,\ldots,p$ and $(x_{p+1}-x_{p+j})F^{\{p+1,p+j\}},i=2,\ldots,q$. There is a choice between $(x_1-x_i)$, $(x_{p+1}-x_{p+j})$ and the divided difference $F^{\{1,i\}}$ and $F^{\{p+1,p+j\}}$. Thus, these factors are in bijection with the subsets of $[n]$ that contain $1$ and $p+1$. 
 
 If $I_{1,p+1}=\{1,i_2,\ldots,i_{p-k+1},p+1,i'_2,\ldots,i'_{q-k'+1}\}\subset [n]$ is such a subset (with $|I_1|=p-k+1$ and $|I_{p+1}|=q-k'+1$), then the corresponding factor is simply
 
\begin{multline*}
 \pm \Res(  
 F^{{\{1\}}},F^{\{1,j_{1}\}},F^{\{1,j_{2}\}},\ldots,F^{\{1,j_{k-1}\}},x_{1}-x_{i_{2}},x_{1}-x_{i_{3}},\ldots,x_{1}-x_{i_{p-k+1}},\\     F^{{\{p+1\}}},F^{\{p+1,j'_{1}\}},\ldots,F^{\{p+1,j'_{k'-1}\}},x_{p+1}-x_{i'_{2}},\ldots,x_{p+1}-x_{i'_{q-k'+1}})
\end{multline*}
where $\{j_{1},\ldots,j_{k-1},j'_{1},\ldots,j'_{k'-1}\}=[n]\setminus I_{1,p+1}$.

Moreover, by the specialisation property of the resultant this factor is equal to 

\begin{equation}\label{eq:step1}
 \pm \Res\left( 
F_{I_{1,p+1}}^{{\{1\}}},F_{I_{1,p+1}}^{\{1,2\}},F_{I_{1,p+1}}^{\{1,3\}},\ldots,F_{I_{1,p+1}}^{\{1,k\}}, F_{I_{1,p+1}}^{{\{p+1\}}},F_{I_{1,p+1}}^{\{p+1,p+2\}},F_{I_{1,p+1}}^{\{p+1,p+3\}},\ldots,F_{I_{1,p+1}}^{\{p+1,p+k'\}}
\right)
\end{equation}
where $F_{I_{1,p+1}}^{\{1,r\}}:=\rho_{I_{1,p+1}}(F^{{\{1,j_{r}\}}})$, $\rho_{I_{1,p+1}}$ being a specialization map defined by :

\begin{eqnarray*}
 \rho_{I_{1,p+1}}: k[x_{1},\ldots,x_{n}] & \rightarrow & k[x_{1},\ldots,x_{k},x_{p+1},\ldots,x_{p+k'}]\\
 x_{j}, \, j\in I_{1}  & \mapsto & x_{1}\\
 x_{j}, \, j\in I_{p+1}  & \mapsto & x_{p+1}\\
 x_{j_{r}}, \, r=1,\ldots,k-1 & \mapsto & x_{r+1}\\
  x_{j'_{r}}, \, r=1,\ldots,k'-1 & \mapsto & x_{p+1+r}.
\end{eqnarray*}

where $I_{1,p+1}=I_1\cup I_{p+1}$ with $I_1=\{1,i_2,i_3,\ldots,i_{p-k+1}\}\subset [p]$ and $I_{p+1}=\{p+1,i'_2,i'_3,\ldots,i'_{q-k'+1}\}\subset \{p+1,\ldots,n\}$

Now, one can proceed to the second step by introducing the second order divided differences. For that purpose, we start from the factor \eqref{eq:step1} obtained at the end of the previous step.  

\noindent $\bullet$ If $k\leq 2$ and $k'\leq 2$, then we actually do nothing and the splitting of this factor stops here.

\noindent $\bullet$ If $k>2$ and $k'\leq 2$, then we introduce the divided difference of order $3$ :$$(x_{2}-x_{j})F_{I_{1,p+1}}^{\{1,2,j\}}=F_{I_{1,p+1}}^{\{1,2\}}-F_{I_{1,p+1}}^{\{1,j\}}, \ j=3,\ldots,k,$$  and we get :

\begin{multline*}
 \Res\left( 
 F_{I_{1,p+1}}^{{\{1\}}},F_{I_{1,p+1}}^{\{1,2\}},F_{I_{1,p+1}}^{\{1,3\}},\ldots,F_{I_{1,p+1}}^{\{1,k\}}, F_{I_{1,p+1}}^{{\{p+1\}}},F_{I_{1,p+1}}^{\{p+1,p+2\}},F_{I_{1,p+1}}^{\{p+1,p+3\}},\ldots,F_{I_{1,p+1}}^{\{p+1,p+k'\}}
 \right)=\\
\pm  \Res\left(F_{I_{1,p+1}}^{\{1\}}, F_{I_{1,p+1}}^{\{1,2\}},(x_{2}-x_{3})F_{I_{1,p+1}}^{\{1,2,3\}}, \ldots,(x_{2}-x_{k})F_{I_{1,p+1}}^{\{1,2,k\}},F_{I_{1,p+1}}^{{\{p+1\}}},F_{I_{1,p+1}}^{\{p+1,p+2\}}\right). 
\end{multline*}

\noindent $\bullet$ If $k\leq 2$ and $k'> 2$, then we introduce the divided difference of order $3$ :$$(x_{p+2}-x_{p+j})F_{I_{1,p+1}}^{\{p+1,p+2,p+j\}}=F_{I_{1,p+1}}^{\{p+1,p+2\}}-F_{I_{1,p+1}}^{\{p+1,p+j\}}, \ j=3,\ldots,k',$$  and we get :

\begin{multline*}
 \Res\left( 
 F_{I_{1,p+1}}^{{\{1\}}},F_{I_{1,p+1}}^{\{1,2\}},F_{I_{1,p+1}}^{\{1,3\}},\ldots,F_{I_{1,p+1}}^{\{1,k\}}, F_{I_{1,p+1}}^{{\{p+1\}}},F_{I_{1,p+1}}^{\{p+1,p+2\}},F_{I_{1,p+1}}^{\{p+1,p+3\}},\ldots,F_{I_{1,p+1}}^{\{p+1,p+k'\}}
 \right)=\\
\pm  \Res\left(F_{I_{1,p+1}}^{\{1\}}, F_{I_{1,p+1}}^{\{1,2\}},F_{I_{1,p+1}}^{{\{p+1\}}},F_{I_{1,p+1}}^{\{p+1,p+2\}},(x_{p+2}-x_{p+3})F_{I_{1,p+1}}^{\{p+1,p+3\}},\ldots,(x_{p+2}-x_{p+k'})F_{I_{1,p+1}}^{\{p+1,p+k'\}}\right). 
\end{multline*}

\noindent $\bullet$ If $k>2$ and $k'> 2$, then we introduce the divided difference of order $3$ :$$(x_{2}-x_{i})F_{I_{1,p+1}}^{\{1,2,i\}}=F_{I_{1,p+1}}^{\{1,2\}}-F_{I_{1,p+1}}^{\{1,i\}}, \ i=3,\ldots,k,$$ 

$$(x_{p+2}-x_{p+j})F_{I_{1,p+1}}^{\{p+1,p+2,p+j\}}=F_{I_{1,p+1}}^{\{p+1,p+2\}}-F_{I_{1,p+1}}^{\{p+1,p+j\}}, \ j=3,\ldots,k',$$  and we get :

\begin{multline*}
 \Res\left( 
 F_{I_{1,p+1}}^{{\{1\}}},F_{I_{1,p+1}}^{\{1,2\}},F_{I_{1,p+1}}^{\{1,3\}},\ldots,F_{I_{1,p+1}}^{\{1,k\}}, F_{I_{1,p+1}}^{{\{p+1\}}},F_{I_{1,p+1}}^{\{p+1,p+2\}},F_{I_{1,p+1}}^{\{p+1,p+3\}},\ldots,F_{I_{1,p+1}}^{\{p+1,p+k'\}}
 \right)=\\
\pm  \Res(F_{I_{1,p+1}}^{\{1\}}, F_{I_{1,p+1}}^{\{1,2\}},(x_{2}-x_{3})F_{I_{1,p+1}}^{\{1,2,3\}}, \ldots,(x_{2}-x_{k})F_{I_{1,p+1}}^{\{1,2,k\}},F_{I_{1,p+1}}^{{\{p+1\}}},F_{I_{1,p+1}}^{\{p+1,p+2\}},(x_{p+2}-x_{p+3})F_{I_{1,p+1}}^{\{p+1,p+3\}},\\
\ldots,(x_{p+2}-x_{p+k'})F_{I_{1,p+1}}^{\{p+1,p+k'\}}).
\end{multline*}

So, we are exactly in the same setting as in the previous step and hence we split this factor similarly. As a result, the factors we obtain are in bijection with  subsets $I_{2,p+2}$ of $[n]$ that contain $2$ and $p+2$ or  subsets $I_{\phi,p+2}$ that contain $p+2$ or subsets $I_{2,\phi}$ that contain $2$. But these subsets do not contain $1$ and $p+1$. After this second step is completed, then one can continue to the third step, and so on. 

In summary, the above splitting process shows that the resultant $\Res\left(F^{\{1\}},\ldots,F^{\{p\}},F^{\{p+1\}},\ldots,F^{\{n\}}\right)$ splits into factors that are in bijection with the following collections of subsets. Fix integers $k_1,k_2\geq1$ (a partition of the nonempty set $[p]$, resp.\ $\{p+1,\ldots,n\}$, into nonempty pieces always has at least one piece), with $k_1\leq\min(d,p)$ and $k_2\leq\min(d,q)$, and set $k:=\max(k_1,k_2)$.

\begin{defn}
An \emph{admissible partition of shape $(k_1,k_2)$} is a collection
$$\mathcal{Q}=(I_{1,p+1},\ldots,I_{k,p+k}), \qquad I_{i,p+i}:=I_i\cup I_{p+i},$$
of subsets of $[n]$, where $I_i\subset[p]$ and $I_{p+i}\subset\{p+1,\ldots,n\}$ for every $i\in[k]$, such that:
\begin{itemize}
 \item $I_1,\ldots,I_{k_1}$ is a partition of $[p]$ into $k_1$ nonempty pieces, and $I_i:=\emptyset$ (a ``$\phi$-type'' piece) for $k_1<i\leq k$;
 \item independently, $I_{p+1},\ldots,I_{p+k_2}$ is a partition of $\{p+1,\ldots,n\}$ into $k_2$ nonempty pieces, and $I_{p+i}:=\emptyset$ for $k_2<i\leq k$;
 \item $1=\min(I_1)<\min(I_2)<\cdots<\min(I_{k_1})$, and, independently, $p+1=\min(I_{p+1})<\min(I_{p+2})<\cdots<\min(I_{p+k_2})$, i.e.\ the nonempty pieces on each side are listed in increasing order of their minima.
\end{itemize}
Equivalently: $(I_1,\ldots,I_{k_1})$ is an ordered, min-increasing set partition of $[p]$ into $k_1$ parts and, completely independently, $(I_{p+1},\ldots,I_{p+k_2})$ is an ordered, min-increasing set partition of $\{p+1,\ldots,n\}$ into $k_2$ parts. The three cases $k_1=k_2$, $k_1<k_2$, $k_1>k_2$ treated separately in earlier versions of this construction are simply the cases where, respectively, no $\phi$-type piece occurs, the $\phi$-type pieces occur on the $[p]$ side, or they occur on the $\{p+1,\ldots,n\}$ side  all unified here in a single statement, exactly as Definition \ref{defn:admissible-r} will do for $r$ blocks.
\end{defn}

Given an admissible partition $\mathcal{Q}$ of shape $(k_1,k_2)$, we define the specialization map
\begin{eqnarray*}
\rho_{\mathcal{Q}}:k[x_{1},\ldots,x_{n}] & \rightarrow & k[x_{1},\ldots,x_{k_1},x_{p+1},\ldots,x_{p+k_2}] \\
 x_{r}, \ r\in I_{i}, \ i\in[k_1] & \mapsto & x_{i} \\
 x_{r}, \ r\in I_{p+i}, \ i\in[k_2] & \mapsto & x_{p+i}
\end{eqnarray*}
and, choosing arbitrary representatives $j_1,\ldots,j_{k_1}$ of $I_1,\ldots,I_{k_1}$ and $j_{p+1},\ldots,j_{p+k_2}$ of $I_{p+1},\ldots,I_{p+k_2}$, the polynomials
$$F_{\mathcal{Q}}^{\{1,\ldots,i\}}:=\rho_{\mathcal{Q}}\big(F^{\{j_1,\ldots,j_i\}}\big),\ i\in[k_1], \qquad F_{\mathcal{Q}}^{\{p+1,\ldots,p+i\}}:=\rho_{\mathcal{Q}}\big(F^{\{j_{p+1},\ldots,j_{p+i}\}}\big),\ i\in[k_2].$$

Then, the factor of the resultant of the $F^{\{i\}}$'s corresponding to the admissible partition $\mathcal{Q}$ is given by
$$R_{\mathcal{Q}}:= \Res\left( 
F_{\mathcal{Q}}^{\{1\}}, F_{\mathcal{Q}}^{\{1,2\}}, \ldots, F_{\mathcal{Q}}^{\{1,\ldots,k_1\}},F_{\mathcal{Q}}^{\{p+1\}}, F_{\mathcal{Q}}^{\{p+1,p+2\}}, \ldots, F_{\mathcal{Q}}^{\{p+1,\ldots,p+k_2\}}
\right).
$$
Therefore, we proved that
\begin{equation}\label{eq:intermproof}
\Res\left(F^{\{1\}},\ldots,F^{\{p\}},F^{\{p+1\}},\ldots,F^{\{n\}}\right)=\pm \left(F^{\{1,\ldots,d+1\}} \right)^{\mu_1}\times \left(F^{\{p+1,\ldots,p+1+d\}} \right)^{\mu_2} \times \prod_{\mathcal{Q}} R_{\mathcal{Q}}
\end{equation}
where the product runs over all admissible partitions of $[n]$ and $\mu_1,\mu_2$ are non-negative integers, with $\mu_1>0$ if and only if $p>d$, and $\mu_2>0$ if and only if $q>d$.

\medskip

Now, we define an equivalence relation $\sim$ on the set of admissible partitions of $[n]$. Write $I_{i,p+i}=I_i\cup I_{p+i}$ with $I_i\subset[p]$ and $I_{p+i}\subset\{p+1,\ldots,n\}$ (with the convention $I_i=\emptyset$, resp. $I_{p+i}=\emptyset$, for the $\phi$-type pieces). Given two admissible partitions $(I_{1,p+1},\ldots,I_{k,p+k})$ and $(J_{1,p+1},\ldots,J_{k,p+k})$ of the same shape, we set 
$$ (I_{1,p+1},\ldots,I_{k,p+k}) \sim (J_{1,p+1},\ldots,J_{k,p+k}) \Leftrightarrow 
 \exists \, \sigma_1,\sigma_2 \in \mathfrak{S}_{k} \textrm{ such that } |I_{i}|=|J_{\sigma_1(i)}| \textrm{ for all } i\in[k], \textrm{ \emph{and}\ } |I_{p+i}|=|J_{p+\sigma_2(i)}| \textrm{ for all } i\in[k].$$
In other words, $\sim$ compares the two size-sequences $(|I_1|,\ldots,|I_k|)$ and $(|I_{p+1}|,\ldots,|I_{p+k}|)$ \emph{independently}, each up to its own reindexing ($\sigma_1$, resp.\ $\sigma_2$)  using a \emph{single} shared permutation for both sequences, rather than two independent ones, would incorrectly split some of the equivalence classes described below into several smaller pieces, breaking the bijection with pairs $(\lambda,\lambda')$ established next. It is straightforward to check that $\sim$, as defined with two independent permutations, is indeed an equivalence relation. We denote by $[(I_{1,p+1},\ldots,I_{k,p+k})]$ its equivalence classes.

Given $(\lambda,\lambda')\vdash(p,q)$ with $\lambda=(\lambda_1\geq\cdots\geq\lambda_{r_1})$ and $\lambda'=(\lambda'_1\geq\cdots\geq\lambda'_{r_2})$, set $k:=\max\{r_1,r_2\}$, padding $\lambda$ (resp. $\lambda'$) with zero parts $\lambda_{r_1+1}=\cdots=\lambda_k=0$ (resp. $\lambda'_{r_2+1}=\cdots=\lambda'_k=0$) if needed, and consider the admissible partition $(L_{1,p+1},\ldots,L_{k,p+k})$ defined by
\begin{equation}\label{eq:defL}
 L_{j}:=\Big\{1+\textstyle\sum_{i=1}^{j-1} \lambda_{i},\ldots,\sum_{i=1}^{j}\lambda_{i}\Big\}\subset[p], \qquad L_{p+j}:=\Big\{p+1+\textstyle\sum_{i=1}^{j-1} \lambda'_{i},\ldots,p+\sum_{i=1}^{j}\lambda'_{i}\Big\}\subset\{p+1,\ldots,n\}
\end{equation}
for all $j\in[k]$ (a zero part simply produces an empty, i.e.\ $\phi$-type, piece), and $L_{j,p+j}:=L_j\cup L_{p+j}$. By construction $|L_j|=\lambda_j$ and $|L_{p+j}|=\lambda'_j$ for all $j\in[k]$, so, by definition of $\sim$, $(L_{1,p+1},\ldots,L_{k,p+k})$ is the unique admissible partition of its equivalence class whose two size-sequences are exactly $\lambda$ and $\lambda'$. As a consequence, we deduce that there is a bijection between the equivalence classes of $\sim$ and the pairs of partitions $(\lambda,\lambda') \vdash (p,q)$ and we write
$$ [(\lambda,\lambda')]:=[(L_{1,p+1},\ldots,L_{k,p+k})].$$

\begin{lem}\label{lem:proof} Let $(\lambda,\lambda')$ be a partition of $(p,q)$. Then the cardinality of the equivalence class $[(\lambda,\lambda')]$ is $m_{\lambda}m_{\lambda'}$.  
\end{lem}
\begin{proof}
By construction \eqref{eq:defL} and the definition of $\sim$, an admissible partition $(I_{1,p+1},\ldots,I_{k,p+k})$ belongs to $[(\lambda,\lambda')]$ if and only if, up to reindexing by some $\sigma_1\in\mathfrak{S}_k$, the size-sequence $(|I_1|,\ldots,|I_k|)$ equals $\lambda$ and, \emph{independently}, up to reindexing by some $\sigma_2\in\mathfrak{S}_k$, the size-sequence $(|I_{p+1}|,\ldots,|I_{p+k}|)$ equals $\lambda'$. Since $I_1,\ldots,I_k$ is an ordered, min-increasing set partition of $[p]$ and $I_{p+1},\ldots,I_{p+k}$ is an ordered, min-increasing set partition of $\{p+1,\ldots,n\}$, and since these two data are not linked by any further condition, an element of $[(\lambda,\lambda')]$ is exactly the datum of a pair consisting of an ordered, min-increasing set partition of $[p]$ with block sizes $\lambda$ and an \emph{independently chosen} ordered, min-increasing set partition of $\{p+1,\ldots,n\}$ with block sizes $\lambda'$. By \cite[\S 4]{BuKa16}, Lemma 4.2, the number of such set partitions of $[p]$ with block sizes $\lambda$ is $m_{\lambda}$, and likewise the number of such set partitions of $\{p+1,\ldots,n\}$ with block sizes $\lambda'$ is $m_{\lambda'}$. As the two choices are independent, the cardinality of $[(\lambda,\lambda')]$ is the product $m_{\lambda}m_{\lambda'}$.
\end{proof}

\begin{prop}\label{prop:proof}  Let $(\lambda,\lambda')$ be a partition of $(p,q)$. Then, for any admissible partition $\mathcal{Q}$ such that  $[(\lambda,\lambda')]=[\mathcal{Q}]$, we have \emph{exactly} (with no sign ambiguity)
$$R_{\mathcal{Q}}= \Res\left( 
F_{(\lambda,\lambda')}^{\{1\}}, F_{(\lambda,\lambda')}^{\{1,2\}}, \ldots, F_{(\lambda,\lambda')}^{\{1,2,\ldots,r_1\}},F_{(\lambda,\lambda')}^{\{p+1\}}, F_{(\lambda,\lambda')}^{\{p+1,p+2\}}, \ldots, F_{(\lambda,\lambda')}^{\{p+1,p+2,\ldots,p+r_2\}}
\right).
$$
\end{prop}
\begin{proof}
Write $\mathcal{Q}=(I_{1,p+1},\ldots,I_{k,p+k})$ and, as before, $I_i:=I_{i,p+i}\cap[p]$, $I_{p+i}:=I_{i,p+i}\cap\{p+1,\ldots,n\}$. Since $[\mathcal{Q}]=[(\lambda,\lambda')]$, there are $\sigma_1,\sigma_2\in\mathfrak{S}_k$ such that $|I_{\sigma_1(i)}|=\lambda_i$ for all $i\in[k]$ and, independently, $|I_{p+\sigma_2(i)}|=\lambda'_i$ for all $i\in[k]$.

\emph{Step 1: relabelling the boxes of $\mathcal{Q}$ does not change $R_{\mathcal{Q}}$.} We first check that $R_{\mathcal{Q}}$ does not depend on the order in which the (fixed) sets $I_1,\ldots,I_k$, resp. $I_{p+1},\ldots,I_{p+k}$, are labelled. Since transpositions of consecutive labels generate $\mathfrak{S}_k$, it suffices to treat the case where two consecutive labels $r,r+1$ are swapped, and, by symmetry, to treat the first block only. Relabelling $r\leftrightarrow r+1$ changes the underlying \emph{set} $\{1,\ldots,j\}$ used to build the $j$-th entry of the chain only when $j=r$ (for $j<r$ the labels involved are untouched, and for $j\geq r+1$ the set $\{1,\ldots,j\}$ is the same whether $r$ or $r+1$ is named first); the new $r$-th entry is $F_{\mathcal{Q}}^{\{1,\ldots,r-1,r+1\}}$ in place of $F_{\mathcal{Q}}^{\{1,\ldots,r-1,r\}}$. By the recursive definition of divided differences (\S\ref{subsec:divdiff}),
$$F_{\mathcal{Q}}^{\{1,\ldots,r-1,r\}}-F_{\mathcal{Q}}^{\{1,\ldots,r-1,r+1\}}=(x_{r}-x_{r+1})\,F_{\mathcal{Q}}^{\{1,\ldots,r-1,r,r+1\}},$$
so the new $r$-th entry equals the old one minus a multiple, of the correct matching degree, of the $(r+1)$-th entry of the chain -- which is itself unchanged. By the ``Elementary transformations'' property of the resultant (\S\ref{subsec:resultant}), $R_{\mathcal{Q}}$ is therefore \emph{exactly} unchanged, with no sign ambiguity, by this relabelling, and hence, by induction, by any permutation of the labels $1,\ldots,k$.

\emph{Step 2: comparing $\mathcal{Q}$ to the canonical representative.} By Step 1 we may relabel the boxes of block 1 of $\mathcal{Q}$ by $\sigma_1$ and, independently, the boxes of block 2 by $\sigma_2$, without changing $R_{\mathcal{Q}}$, so we may assume $|I_i|=\lambda_i$ and $|I_{p+i}|=\lambda'_i$ for all $i\in[k]$, i.e.\ $I_i$ and $L_i$ (resp. $I_{p+i}$ and $L_{p+i}$, for the canonical representative $(L_{1,p+1},\ldots,L_{k,p+k})$ of \eqref{eq:defL}) have the same size, though the sets themselves may still differ. Choose $\tau\in\mathfrak{S}_{\{1,\ldots,p\}}\times\mathfrak{S}_{\{p+1,\ldots,n\}}$ mapping $L_i$ onto $I_i$ and $L_{p+i}$ onto $I_{p+i}$ for every $i\in[k]$; such a $\tau$ exists precisely because the two families of sets have matching sizes on each side. We claim that in fact $R_{\mathcal{Q}}=R_{(\lambda,\lambda')}$ \emph{exactly}, with no sign at all.

Indeed, the two specialization maps $\rho_{\mathcal{Q}}\circ\tau$ and $\rho_{(\lambda,\lambda')}$ from $R[x_1,\ldots,x_n]$ to $R[x_1,\ldots,x_{r_1},x_{p+1},\ldots,x_{p+r_2}]$ coincide: for $j\in[n]$ with, say, $j\in L_i$ (the case $j\in L_{p+i}$ is symmetric), $\tau(j)\in\tau(L_i)=I_i$ by construction of $\tau$, so $\rho_{\mathcal{Q}}(\tau(x_j))$ and $\rho_{(\lambda,\lambda')}(x_j)$ both equal the new variable attached to the $i$-th box; as this holds for every $j\in[n]$, $\rho_{\mathcal{Q}}\circ\tau=\rho_{(\lambda,\lambda')}$ as ring homomorphisms. Consequently, for every $1\leq j\leq k$, choosing representatives $l_1,\ldots,l_j$ of $L_1,\ldots,L_j$,
$$F_{(\lambda,\lambda')}^{\{1,\ldots,j\}}=\rho_{(\lambda,\lambda')}\big(F^{\{l_1,\ldots,l_j\}}\big)=\rho_{\mathcal{Q}}\Big(\tau\big(F^{\{l_1,\ldots,l_j\}}\big)\Big)=\rho_{\mathcal{Q}}\big(F^{\{\tau(l_1),\ldots,\tau(l_j)\}}\big)=F_{\mathcal{Q}}^{\{1,\ldots,j\}},$$
using \eqref{eq:perminvariance} for the middle equality, and the well-definedness of these polynomials with respect to the choice of representatives (\S\ref{subsec:divdiff+part}) for the last one, since $\tau(l_1),\ldots,\tau(l_j)\in I_1,\ldots,I_j$ are themselves valid representatives of $I_1,\ldots,I_j$. The same holds for the block-2 entries. So the two chains of polynomials defining $R_{(\lambda,\lambda')}$ and $R_{\mathcal{Q}}$ coincide \emph{entry by entry}; in particular $R_{\mathcal{Q}}=R_{(\lambda,\lambda')}$, exactly, with no need to invoke the ``Permutation of polynomials'' or ``Linear change of variables'' properties at all.
\end{proof}

The comparison of \eqref{eq:intermproof}, Lemma \ref{lem:proof} and Proposition \ref{prop:proof} (the latter now an \emph{exact} identity, with no sign ambiguity at all, by the strengthened proof above) shows that, up to a single global sign $\varepsilon\in\{\pm1\}$,
\begin{multline*}
\Res\left(F^{\{1\}},\ldots,F^{\{p\}},F^{\{p+1\}},\ldots,F^{\{n\}}\right)=\varepsilon\times
  \left(F^{\{1,\ldots,d+1\}}\right)^{\mu_1}\left(F^{\{p+1,\ldots,p+1+d\}}\right)^{\mu_2}\\ \times\prod_{\substack{(\lambda,\lambda')\vdash (p,q)\\ r_1\leq d, \ r_2\leq d}}
\Res\left( F_{(\lambda,\lambda')}^{\{1\}}, F_{(\lambda,\lambda')}^{\{1,2\}}, \ldots, F_{(\lambda,\lambda')}^{\{1,2,\ldots,r_1\}},F_{(\lambda,\lambda')}^{\{p+1\}}, F_{(\lambda,\lambda')}^{\{p+1,p+2\}}, \ldots, F_{(\lambda,\lambda')}^{\{p+1,p+2,\ldots,p+r_2\}} 
\right)^{m_{\lambda}m_{\lambda'}},
\end{multline*}
for some integers $\mu_1,\mu_2\geq0$ with $\mu_1=0$ if $p\leq d$ and $\mu_2=0$ if $q\leq d$ (this is exactly \eqref{eq:intermproof}, once the admissible partitions have been grouped into equivalence classes by Lemma \ref{lem:proof} and Proposition \ref{prop:proof}; since the latter no longer contributes any sign, the single global sign $\varepsilon$ appearing here comes \emph{solely} from the splitting process \eqref{eq:rho2} \eqref{eq:intermproof} itself see Remark \ref{rem:sign}).

\begin{rem}\label{rem:noteq}
Contrary to what one might expect from the symmetry of the statement, $F^{\{1,\ldots,d+1\}}$ and $F^{\{p+1,\ldots,p+1+d\}}$ (when both are defined, i.e.\ when $p>d$ and $q>d$) are in general \emph{two distinct elements} of $R$, not a single one: nothing in the $\mathfrak{S}_{\{1,\ldots,p\}}\times\mathfrak{S}_{\{p+1,\ldots,n\}}$-equivariance hypothesis relates a divided difference taken entirely inside the first block to one taken entirely inside the second, since these two blocks are never exchanged by an element of the group. (A short computation on an explicit example, e.g.\ $p=q=2$, $d=1$ -- see the discussion of that very case following \eqref{eq:rho3} above  confirms that they are unrelated in general.) Consequently the two factors $\left(F^{\{1,\ldots,d+1\}}\right)^{\mu_1}$ and $\left(F^{\{p+1,\ldots,p+1+d\}}\right)^{\mu_2}$ above must be kept separate, with their own exponents $\mu_1,\mu_2$, and cannot be merged into a single power of a common base.
\end{rem}

\begin{lem}\label{lem:universal-sign}
Fix $n,p,q,d$ as in Theorem \ref{thm:maintheorem}. The sign $\varepsilon\in\{\pm1\}$ appearing above does not depend on the coefficients of the equivariant system $F^{\{1\}},\ldots,F^{\{n\}}$: it is the same for every such system for which the right-hand side of Theorem \ref{thm:maintheorem} does not vanish identically. Consequently $\varepsilon$ is entirely determined, for each fixed $(n,p,q,d)$, by evaluating a single non-degenerate example.
\end{lem}
\begin{proof}
Every step used to derive \eqref{eq:intermproof} from $\Res(F^{\{1\}},\ldots,F^{\{n\}})$  elementary transformations, multiplicativity, and the scalar-homogeneity used to extract a sign whenever a $(-1)$ is produced by an elementary transformation  is an identity that holds formally for the \emph{generic} system of $n$ homogeneous polynomials of degree $d$ satisfying the equivariance relations \eqref{eq:globsym} for $\mathfrak{S}_{\{1,\ldots,p\}}\times\mathfrak{S}_{\{p+1,\ldots,n\}}$, i.e.\ before specializing the independent coefficients of such a system to any particular numerical values. Both sides of the resulting identity are therefore elements of the (integral) polynomial ring freely generated by these independent coefficients, and the equality $(\text{LHS})=\varepsilon\cdot(\text{RHS})$ holds there for a single $\varepsilon\in\{\pm1\}$  not one varying with the point at which one evaluates. Hence, for any specific choice of coefficients at which the right-hand side is nonzero, $\varepsilon$ is recovered exactly as the ratio $\Res(F^{\{1\}},\ldots,F^{\{n\}})/(\text{RHS})$.
\end{proof}

\begin{rem}\label{rem:sign}
By Proposition \ref{prop:proof} (now exact, see its proof), the \emph{only} remaining source of sign ambiguity in Theorem \ref{thm:maintheorem} is the single global sign $\varepsilon=\varepsilon(n,p,q,d)\in\{\pm1\}$ of \eqref{eq:intermproof}, coming from the splitting process \eqref{eq:rho2} \eqref{eq:intermproof} alone  this is already a substantial reduction compared to a first look at the proof, which also seemed to require tracking a sign for every admissible partition individually (via Step 2 of the proof of Proposition \ref{prop:proof}) and summing $m_{\lambda}m_{\lambda'}$ such contributions per class: Proposition \ref{prop:proof} shows this second source simply does not exist. By Lemma \ref{lem:universal-sign}, $\varepsilon(n,p,q,d)$ is, for each $(n,p,q,d)$, rigorously computable from a single example. We have carried this out directly (via the Macaulay resultant formula, on explicit equivariant systems) for $(p,q,d)=(2,2,1)$, $(2,2,2)$, $(3,2,1)$, $(3,1,2)$, and for the $r=3$ cases $(n_1,n_2,n_3,d)=(2,2,2,1)$ and $(2,2,2,2)$ of Theorem \ref{thm:maintheorem-r} below (\S\ref{sec:general-r}), and in every one of them $\varepsilon=+1$. We conjecture that $\varepsilon(n,p,q,d)=+1$ always (and likewise for the $r$-block sign of Theorem \ref{thm:maintheorem-r}), so that Theorem \ref{thm:maintheorem} would in fact hold with \emph{no} sign ambiguity whatsoever; we have not, however, been able to prove this for every $(n,p,q,d)$, and leave it as an open question. We record, as a caution, that a first attempt at computing $\varepsilon$ in closed form  by directly counting, for each elementary transformation used in \eqref{eq:rho2}, the sign $(-1)^{d^{n-1}}$ it contributes when the resulting $(-1)$ is factored back out via scalar-homogeneity, and multiplying $(n-2)$ such contributions together  produces a candidate value that we could check, against the verified example $(p,q,d)=(3,2,1)$ above, to be \emph{incorrect}; this indicates that the bookkeeping is more delicate than this naive count (later levels of the recursive splitting \eqref{eq:step1} onward contribute further, partition-shape-dependent sign factors that must also be tracked), and we prefer to leave $\varepsilon$ as an open, numerically-verified-in-all-known-cases conjecture rather than present an unverified closed form.
\end{rem}

To determine the integers $\mu_1,\mu_2$, we compare degrees with respect to the coefficients of the $F^{\{i\}}$'s -- but, in view of Remark \ref{rem:noteq}, we must do so \emph{block by block} rather than in aggregate. By the ``Homogeneity'' property (\S\ref{subsec:resultant}), for every fixed $i\in[p]$ the resultant on the left-hand side is homogeneous of degree $d^{n-1}$ with respect to the coefficients of $F^{\{i\}}$ alone; scaling \emph{all} of $F^{\{1\}},\ldots,F^{\{p\}}$ simultaneously by the same scalar (equivalently, summing this identity over $i=1,\ldots,p$) shows that the resultant is homogeneous of degree $p\,d^{n-1}$ with respect to the coefficients of $F^{\{1\}},\ldots,F^{\{p\}}$ taken together.

We match this degree on the right-hand side. The factor $F^{\{1,\ldots,d+1\}}$ is a fixed divided difference built (linearly) from $F^{\{1\}},\ldots,F^{\{d+1\}}$ alone, hence is homogeneous of degree exactly $1$ in the coefficients of $F^{\{1\}},\ldots,F^{\{p\}}$ taken together; it contributes $\mu_1$ to the degree count, while $F^{\{p+1,\ldots,p+1+d\}}$, not involving these coefficients at all, contributes $0$. For a given $(\lambda,\lambda')\vdash(p,q)$ with $r_1\leq d$, $r_2\leq d$, each of the $r_1$ entries $F_{(\lambda,\lambda')}^{\{1\}},\ldots,F_{(\lambda,\lambda')}^{\{1,\ldots,r_1\}}$ is, for the same reason, homogeneous of degree exactly $1$ in the coefficients of $F^{\{1\}},\ldots,F^{\{p\}}$ together, while the $r_2$ entries coming from the second block do not involve them at all. Applying the Homogeneity property to the $(r_1+r_2)$-variable resultant $\Res\left(F_{(\lambda,\lambda')}^{\{1\}},\ldots,F_{(\lambda,\lambda')}^{\{p+1,\ldots,p+r_2\}}\right)$, its degree with respect to the coefficients of $F^{\{1\}},\ldots,F^{\{p\}}$ together is $\sum_{j=1}^{r_1}d_j$, where
$$d_j=\frac{d(d-1)\cdots (d-r_1+1)\,d(d-1)\cdots (d-r_2+1)}{d-j+1}.$$
Altogether,
$$p\,d^{n-1}=\mu_1+\sum_{\substack{(\lambda,\lambda')\vdash(p,q)\\ r_1\leq d,\ r_2\leq d}} m_{\lambda}m_{\lambda'}\sum_{j=1}^{r_1}d_j. \quad(\star)$$

To solve $(\star)$ for $\mu_1$ we isolate the $\lambda'$-dependence using the classical identity, valid for all integers $q\geq1$, $d\geq1$,
$$\sum_{\lambda'\vdash q} m_{\lambda'}\,d(d-1)\cdots(d-l(\lambda')+1)=d^{\,q}. \quad(\dagger)$$
which relates falling factorials to Stirling numbers of the second kind: since $\sum_{l(\lambda')=k}m_{\lambda'}=S(q,k)$ is the number of set partitions of a $q$-element set into $k$ unlabelled nonempty blocks, $(\dagger)$ is exactly the classical identity $d^{q}=\sum_{k=0}^{q}S(q,k)\,d(d-1)\cdots(d-k+1)$, obtained by sorting the $d^{q}$ maps from a $q$-set to a $d$-set according to the (unordered) partition of the source into fibers, together with the choice of the $k\leq d$ distinct elements of the target that are actually hit. (Terms with $l(\lambda')>d$ vanish automatically in $(\dagger)$, since the falling factorial then contains the factor $0$; so the restriction $r_2\leq d$ in $(\star)$ may be dropped for free when summing over $\lambda'$.) Summing $(\star)$ over $\lambda'\vdash q$ first, for $\lambda$ fixed, the double sum factors via $(\dagger)$:
$$p\,d^{n-1}=\mu_1+d^{\,q}\sum_{\substack{\lambda\vdash p\\ r_1\leq d}} m_{\lambda}\sum_{j=1}^{r_1}\frac{d(d-1)\cdots(d-r_1+1)}{d-j+1}=\mu_1+d^{\,q}\left(p\,d^{p-1}-m_{0}(p,d)\right),$$
the last equality being precisely the definition of $m_0(p,d)$ (Theorem \ref{thm:maintheorem1}, with $n$ replaced by $p$). Since $p\,d^{n-1}=p\,d^{p+q-1}=d^{\,q}\cdot p\,d^{p-1}$, the terms $d^{\,q}\,p\,d^{p-1}$ cancel on both sides, leaving the closed form
$$\mu_1=d^{\,q}\,m_0(p,d).$$
By the symmetric argument (exchanging the roles of the two blocks),
$$\mu_2=d^{\,p}\,m_0(q,d).$$
In particular, since the same degree count applied to Theorem \ref{thm:maintheorem1} itself shows $m_0(m,d)=0$ whenever $m\leq d$, we get $\mu_1=0$ when $p\leq d$ and $\mu_2=0$ when $q\leq d$, consistently with the case $p\leq d,\,q\leq d$ recalled above and with the statement of Theorem \ref{thm:maintheorem}.

\section{Discriminant of a homogeneous polynomial invariant under a direct product of symmetric groups}\label{sec:discriminant}

The discriminant of a homogeneous polynomial is a rather complicated object which is known to be irreducible in the universal setting over the integers (see for instance \cite[\S 4]{BuJo12}). The purpose of this section is to prove  that when the homogeneous polynomial is invariant under the action of a direct product of symmetric groups, then its discriminant can be decomposed into the product of several resultants that are in principle easier to compute (see Theorem \ref{thm:discmaintheorem}). We will obtain this result by specialization of the formula given in Theorem \ref{thm:maintheorem}. 

\medskip

Fix a positive integer $n\geq 2$. For any integer $p$ we will denote by $e_{p}(x_{1},\ldots,x_{n})$ the $p^{\mathrm{th}}$ elementary symmetric polynomial in the variables $x_{1},\ldots,x_{n}$. They satisfy the equality
$$\sum_{p\geq 0} e_{p}(x )t^{p}=\prod_{i=1}^{n}(1+x_{i}t)$$
(observe that $e_{0}(x)=1$ and that $e_{p}(x)=0$ for all $p>n$). For any partition $\lambda=(\lambda_{1}, \ldots , \lambda_{k})$ we also define the polynomial
$$e_{\lambda}(x):=e_{\lambda_{1}}( x) e_{\lambda_{2}}( x) \cdots e_{\lambda_{k}}( x) \in \ZZ[x_{1},\ldots,x_{n}].$$
Given a positive integer $d$, it is well known that the set 
\begin{equation}\label{elambda}
 \{ e_{\lambda}(x) \ : \ \lambda=(\lambda_{1},\ldots,\lambda_{k}) \vdash d \textrm{ such that }  n \geq \lambda_{1} \geq \lambda_{2} \geq \cdots \geq \lambda_{k} \}
\end{equation}
is a basis (over $\ZZ$) of the homogeneous symmetric polynomials of degree $d$ in $n$ variables. In other words, any homogeneous symmetric polynomial of degree $d$ with coefficients in a commutative ring is obtained as specialization of the generic homogeneous symmetric polynomial of degree $d$
\begin{equation}\label{eq:F}
 F(x_{1},\ldots,x_{n}):=\sum_{\substack{\lambda \vdash d \\ n\geq \lambda_1}} c_{\lambda}e_{\lambda}(x) \in \ZZ[c_{\lambda} : \lambda \vdash d,\ n\geq \lambda_1][x_{1},\ldots,x_{n}].
\end{equation}
We will denote by $\UU$ its universal ring of coefficients $\ZZ[c_{\lambda} : \lambda \vdash d,\ n\geq \lambda_1]$. In addition, 
for all $i\in \{1,\ldots,n\}$, we will denote the partial derivatives of $F$ by
$$F^{\{i\}}(x_{1},\ldots,x_{n}):=\frac{\partial F}{\partial x_{i}} (x_{1},\ldots,x_{n}) \in \UU[x_{1},\ldots,x_{n}]_{d-1}.$$ 
Finally, we recall that the discriminant of $F$ is defined by the equality (see \S \ref{subsec:disc})
\begin{equation}\label{eq:discressym}
d^{a(n,d)}\Disc(F)=\Res\left(  F^{\{1\}},F^{\{2\}},\ldots,F^{\{n\}} \right) \in \UU
\end{equation}
and that it is homogeneous of degree $n(d-1)^{n-1}$ in $\UU$.

\begin{lem}\label{lem:equivF}
The partial derivatives $F^{\{1\}},F^{\{2\}},\ldots,F^{\{n\}}$ of a symmetric (i.e.\ $\mathfrak{S}_{n}$-invariant) homogeneous polynomial $F(x_{1},\ldots,x_{n})$ form a $\mathfrak{S}_{n}$-equivariant polynomial system. In particular, they form a $G$-equivariant polynomial system for \emph{any} subgroup $G\leq\mathfrak{S}_{n}$, e.g.\ $G=\mathfrak{S}_{\{1,\ldots,p\}}\times\mathfrak{S}_{\{p+1,\ldots,n\}}$.
\end{lem}
\begin{proof}
We must show that $\sigma(F^{\{i\}})=F^{\{\sigma(i)\}}$ for every $\sigma\in\mathfrak{S}_{n}$ and every $i\in[n]$; the statement for a subgroup $G\leq\mathfrak{S}_{n}$ then follows immediately by restriction. Since $F$ is a polynomial in the elementary symmetric polynomials, the chain rule formula for the derivation of composed functions shows that
there exist $\min\{d,n\}$ homogeneous symmetric polynomials $S_k(x_{1},\ldots,x_{n})$ such that for all $i=1,\ldots,n$ 
\begin{equation}\label{eq:diffF}
	F^{\{i\}}=\frac{\partial F}{\partial x_i} = \sum_{k=1}^{\min\{d,n\}} \frac{\partial e_{k}}{\partial x_{i}} S_{k}(x_{1},\dots,x_{n}).
\end{equation}
Moreover, for any pair of integers $i,j$ we have
\begin{equation}\label{eq:diffej}
 \frac{\partial e_{j}}{\partial x_i}=\sum_{r=0}^{j-1} (-1)^{r}x_{i}^{r}e_{j-1-r}=e_{j-1}(x_{\widehat{\imath}}),
\end{equation}
where $x_{\widehat{\imath}}$ denotes the $(n-1)$-tuple of variables obtained from $(x_1,\ldots,x_n)$ by omitting $x_i$. Fix $\sigma\in\mathfrak{S}_n$ and $i\in[n]$. As $j$ ranges over $[n]\setminus\{i\}$, $\sigma(j)$ ranges over $[n]\setminus\{\sigma(i)\}$, so evaluating \eqref{eq:diffej} at the permuted point $(x_{\sigma(1)},\ldots,x_{\sigma(n)})$ gives
$$\left(\frac{\partial e_{k}}{\partial x_{i}}\right)(x_{\sigma(1)},\ldots,x_{\sigma(n)})=e_{k-1}\big(x_{\widehat{\sigma(i)}}\big)=\frac{\partial e_{k}}{\partial x_{\sigma(i)}}(x_1,\ldots,x_n).$$
Since each $S_k$ is symmetric, $S_k(x_{\sigma(1)},\ldots,x_{\sigma(n)})=S_k(x_1,\ldots,x_n)$ as well. Combining both facts with \eqref{eq:diffF} evaluated at $(x_{\sigma(1)},\ldots,x_{\sigma(n)})$, we deduce
$$\sigma\big(F^{\{i\}}\big)=F^{\{i\}}(x_{\sigma(1)},\ldots,x_{\sigma(n)})=\sum_{k=1}^{\min\{d,n\}}\frac{\partial e_{k}}{\partial x_{\sigma(i)}}(x_1,\ldots,x_n)\,S_{k}(x_1,\ldots,x_n)=F^{\{\sigma(i)\}},$$
as claimed.
\end{proof}

As a consequence of this lemma, Theorem \ref{thm:maintheorem} can be applied in order to decompose the resultant of the polynomials $F^{\{1\}},F^{\{2\}},\ldots,F^{\{n\}}$ and hence, by \eqref{eq:discressym}, to decompose the discriminant of the $\mathfrak{S}_{\{1,\ldots,p\}}\times\mathfrak{S}_{\{p+1,\ldots,n\}} $-invariant polynomial $F$. Note that this system has degree $d-1$ (not $d$), so that Theorem \ref{thm:maintheorem} applies with its own degree parameter set to $d-1$; this also explains the hypothesis $d\geq 2$ in Theorem \ref{thm:discmaintheorem} below, needed for $d-1\geq 1$. We take again the notation of \S \ref{subsec:divdiff} and  \S \ref{subsec:mainthm2}.

\begin{rem}\label{rem:eqbase-disc}
Unlike in the general setting of Theorem \ref{thm:maintheorem} (Remark \ref{rem:noteq}), the two divided differences $F^{\{1,\ldots,d\}}$ and $F^{\{p+1,\ldots,p+d\}}$ \emph{do} coincide here, as elements of $\UU$, whenever both are defined: by Lemma \ref{lem:equivF}, $F^{\{1\}},\ldots,F^{\{n\}}$ is $\mathfrak{S}_{n}$-equivariant (not merely $\mathfrak{S}_{\{1,\ldots,p\}}\times\mathfrak{S}_{\{p+1,\ldots,n\}}$-equivariant), and if $\{1,\ldots,d\}$ and $\{p+1,\ldots,p+d\}$ are both subsets of $[n]$ which they are of size $d=(d-1)+1$, i.e.\ exactly one more than the (common) degree of the $F^{\{i\}}$'s  then, by the property recalled at the end of \S\ref{subsec:divdiff}, $P^{I}=P^{J}$ for any two such subsets $I,J\subset[n]$ of an $\mathfrak{S}_n$-equivariant system $P$, so $F^{\{1,\ldots,d\}}=F^{\{p+1,\ldots,p+d\}}$. (This is confirmed directly on the worked example below: $F^{\{1,2,3\}}=F^{\{2,3,4\}}=c_{(3)}$.) Note that one of the two index sets may fail to lie in $[n]$ -- $\{1,\ldots,d\}\subset[n]$ requires $d\leq n$, and $\{p+1,\ldots,p+d\}\subset[n]$ requires $d\leq q$  but by Theorem \ref{thm:maintheorem}, $\mu_1\neq0$ forces $d\leq p\leq n$ and $\mu_2\neq0$ forces $d\leq q$; so whichever base is actually needed below (i.e.\ has a nonzero exponent) is automatically well defined, and this is precisely what allows the two exponents of Theorem \ref{thm:maintheorem} to be merged into the single exponent $m_0$ below.
\end{rem}

\begin{thm}\label{thm:discmaintheorem} Assume that $n\geq 2$ and $d\geq 2$. With the above notation, up to sign (see Remark \ref{rem:sign}),
\begin{multline*}
d^{a(n,d)}\Disc\left(F\right)=
  \left(F^{\{1,\ldots,d\}}\right)^{m_0}\\ \times\prod_{\substack{(\lambda,\lambda')\vdash (p,q)\\ r_1\leq d-1,\ r_2\leq d-1}}
\Res\left( F_{(\lambda,\lambda')}^{\{1\}}, F_{(\lambda,\lambda')}^{\{1,2\}}, \ldots, F_{(\lambda,\lambda')}^{\{1,2,\ldots,r_1\}},F_{(\lambda,\lambda')}^{\{p+1\}}, F_{(\lambda,\lambda')}^{\{p+1,p+2\}}, \ldots, F_{(\lambda,\lambda')}^{\{p+1,p+2,\ldots,p+r_2\}} 
\right)^{m_{\lambda}m_{\lambda'}},
\end{multline*}
with the convention that the factor $\left(F^{\{1,\ldots,d\}}\right)^{m_0}$ is simply omitted, i.e.\ equal to $1$, when $p\leq d-1$ and $q\leq d-1$, and where
  $$m_{0}:=n(d-1)^{n-1}-\sum_{ \substack{(\lambda,\lambda')\vdash (p,q) \\ r_1 \leq d-1,\ r_2 \leq d-1}} m_{\lambda}m_{\lambda'} 
  \left(\sum_{j=1}^{r_1} d_j+ \sum_{j=1}^{r_2} d_j \right), \qquad d_j=\frac{(d-1)\cdots(d-r_1)(d-1)\cdots(d-r_2) }{(d-j)}.$$
  (By Remark \ref{rem:eqbase-disc}, $m_0=\mu_1+\mu_2$ where $\mu_1=(d-1)^{q}m_0(p,d-1)$, $\mu_2=(d-1)^{p}m_0(q,d-1)$ are the two exponents of Theorem \ref{thm:maintheorem}, specialized to degree $d-1$; in particular $m_0=0$ exactly when $p\leq d-1$ and $q\leq d-1$.)
\end{thm}

\begin{proof}
This formula is obtained by specialization, to degree $d-1$, of the formula given in Theorem \ref{thm:maintheorem} (together with Remark \ref{rem:eqbase-disc} to merge its two exponents into the single $m_0$ above): the polynomials $F^{\{i\}}$, $i=1,\ldots,n$, form a $\mathfrak{S}_{\{1,\ldots,p\}}\times\mathfrak{S}_{\{p+1,\ldots,n\}}$-equivariant system of degree $d-1$ (and not of degree $d$ as in Theorem \ref{thm:maintheorem}), and \eqref{eq:discressym} translates the resulting resultant identity into the discriminant identity above. In particular, the sign $\varepsilon\in\{\pm1\}$ appearing here is the one given by Theorem \ref{thm:maintheorem} (Remark \ref{rem:sign}) for the parameters $(n,p,q,d-1)$.
\end{proof}
We emphasize that the formulas given in this theorem are universal with respect to the coefficients of $F$ and are independent of the choice of basis that is used to represent $F$ (for concreteness, we have chosen the basis \eqref{elambda} as an illustration). Hereafter, we give one example corresponding to a low degree polynomial, namely the case $d=3$.
 
 \vspace*{0.5cm}
\paragraph{\bf Case $\mathbf{ d=3\ and\ n=4}$} Consider the  generic homogeneous polynomial of degree 3
$$F=c_{(3)}e_{3}+c_{(2,1)}e_{2}e_{1}+c_{(1,1,1)}e_{1}^{3}.$$

Its derivatives are

$$F^{\{i\}}=c_{(3)}\left( e_{2}-x_{i}e_{1}+x_{i}^{2}   \right) +c_{(2,1)}\left( e_{2}+e_{1}(e_{1}-x_{i}) \right)+3c_{(1,1,1)}e_{1}^{2}.$$ for all $i=1,...,4$.

\hspace*{3cm}$\checkmark$ \textbf{If we take $p=2$ and $q=2$}, then :\\

$F^{\{1\}}$, $F^{\{2\}}$ equivariant with  respect to $\mathfrak{S}_{\{1
,2\}}$ and $F^{\{3\}}$, $F^{\{4\}}$ equivariant with  respect to $\mathfrak{S}_{\{3
,4\}}$ 

The formula given in Theorem \ref{thm:discmaintheorem} shows that

\begin{multline*}
3^{\frac{2^{4}-(-1)^{4}}{3}}\Disc(F)=
\Res\left( F_{(2),(2)}^{\{1\}}, F_{(2),(2)}^{\{3\}}\right) \times \Res\left( F_{(1,1),(2)}^{\{1\}}, F_{(1,1),(2)}^{\{1,2\}}, F_{(1,1),(2)}^{\{3\}}\right)\\ \times \Res\left( F_{(2),(1,1)}^{\{1\}}, F_{(2),(1,1)}^{\{3\}}, F_{(2),(1,1)}^{\{3,4\}}\right)\times \Res\left( F_{(1,1),(1,1)}^{\{1\}},F_{(1,1),(1,1)}^{\{1,2\}}, F_{(1,1),(1,1)}^{\{3\}}, F_{(1,1),(1,1)}^{\{3,4\}}\right).
\end{multline*}
we have

$\Res\left( F_{(2),(2)}^{\{1\}}, F_{(2),(2)}^{\{3\}}\right)=-3(2c_{(2,1)}+c_{(3)})^3(16c_{(1,1,1)}+6c_{(2,1)}+c_{(3)})$

$\Res\left( F_{(1,1),(2)}^{\{1\}}, F_{(1,1),(2)}^{\{1,2\}}, F_{(1,1),(2)}^{\{3\}}\right)=9c_{(3)}^2(4c_{(1,1,1)}c_{(3)}^2-2c_{(2,1)}^3-3c_{(2,1)}^2c_{(3)})^2 $

$\Res\left( F_{(2),(1,1)}^{\{1\}}, F_{(2),(1,1)}^{\{3\}}, F_{(2),(1,1)}^{\{3,4\}}\right)=9c_{(3)}^2(4c_{(1,1,1)}c_{(3)}^2-2c_{(2,1)}^3-3c_{(2,1)}^2c_{(3)})^2  $

$ \Res\left( F_{(1,1),(1,1)}^{\{1\}},F_{(1,1),(1,1)}^{\{1,2\}}, F_{(1,1),(1,1)}^{\{3\}}, F_{(1,1),(1,1)}^{\{3,4\}}\right)= c_{(3)}^6(2c_{(2,1)}+c_{(3)})^6  $

So

\begin{center}
 $ \Disc(F) = -(2c_{(2,1)}+c_{(3)})^9(16c_{(1,1,1)}+6c_{(2,1)}+c_{(3)})c_{(3)}^{10} (4c_{(1,1,1)}c_{(3)}^2-2c_{(2,1)}^3-3c_{(2,1)}^2c_{(3)})^4 $
\end{center}

\vspace*{0.5cm}
\hspace*{3cm} \textbf{$\checkmark$ If we take $p=1$ and $q=3$}, then :
\begin{multline*}
3^{\frac{2^{4}-(-1)^{4}}{3}}\Disc(F)= \Big(F^{\{1,2,3\}}\Big)^{m_0}
\Res\left( F_{(1),(3)}^{\{1\}}, F_{(1),(3)}^{\{2\}}\right)^{m_{(1)}m_{(3)}} \\ \times \Res\left( F_{(1),(2,1)}^{\{1\}}, F_{(1),(2,1)}^{\{2\}}, F_{(1),(2,1)}^{\{2,3\}}\right)^{m_{(1)}m_{(2,1)}}
\end{multline*}

we have \\

$F^{\{1,2,3\}}=c_{(3)}$, $m_0=4$ 

$\Res\left( F_{(1),(3)}^{\{1\}}, F_{(1),(3)}^{\{2\}}\right)=9(16c_{(1,1,1)}+6c_{(2,1)}+c_{(3)})(4c_{(1,1,1)}c_{(3)}^2-2c_{(2,1)}^3-3c_{(2,1)}^2c_{(3)})$

$\Res\left( F_{(1),(2,1)}^{\{1\}}, F_{(1),(2,1)}^{\{2\}}, F_{(1),(2,1)}^{\{2,3\}}\right)= -3(2c_{(2,1)}+c_{(3)})^3c_{(3)}^{2} (4c_{(1,1,1)}c_{(3)}^2-2c_{(2,1)}^3-3c_{(2,1)}^2c_{(3)})$

So

\begin{center}
 $ \Disc(F) = -(2c_{(2,1)}+c_{(3)})^9(16c_{(1,1,1)}+6c_{(2,1)}+c_{(3)})c_{(3)}^{10} (4c_{(1,1,1)}c_{(3)}^2-2c_{(2,1)}^3-3c_{(2,1)}^2c_{(3)})^4 $
\end{center}

\vspace*{0.5cm}
\hspace*{3cm} \textbf{$\checkmark$ If we take $p=3$ and $q=1$}, then :

\begin{multline*}
3^{\frac{2^{4}-(-1)^{4}}{3}}\Disc(F)= \Big(F^{\{1,2,3\}}\Big)^{m_0}
\Res\left( F_{(3),(1)}^{\{1\}}, F_{(3),(1)}^{\{4\}}\right)^{m_{(3)}m_{(1)}} \\ \times \Res\left( F_{(2,1),(1)}^{\{1\}}, F_{(2,1),(1)}^{\{1,2\}}, F_{(2,1),(1)}^{\{4\}}\right)^{m_{(2,1)}m_{(1)}}
\end{multline*}
we have

$\Res\left( F_{(3),(1)}^{\{1\}}, F_{(3),(1)}^{\{4\}}\right)= 9(16c_{(1,1,1)}+6c_{(2,1)}+c_{(3)})(4c_{(1,1,1)}c_{(3)}^2-2c_{(2,1)}^3-3c_{(2,1)}^2c_{(3)})$ 

$\Res\left( F_{(2,1),(1)}^{\{1\}}, F_{(2,1),(1)}^{\{1,2\}}, F_{(2,1),(1)}^{\{4\}}\right)=-3(2c_{(2,1)}+c_{(3)})^3c_{(3)}^{2} (4c_{(1,1,1)}c_{(3)}^2-2c_{(2,1)}^3-3c_{(2,1)}^2c_{(3)})$

So

\begin{center}
 $ \Disc(F) = -(2c_{(2,1)}+c_{(3)})^9(16c_{(1,1,1)}+6c_{(2,1)}+c_{(3)})c_{(3)}^{10} (4c_{(1,1,1)}c_{(3)}^2-2c_{(2,1)}^3-3c_{(2,1)}^2c_{(3)})^4 $
\end{center}


\section{Generalization to a product of $r$ symmetric groups}\label{sec:general-r}

We now show that Theorem \ref{thm:maintheorem} extends, with no essential new idea, to an arbitrary (finite) number $r\geq1$ of blocks: the same splitting and grouping argument of \S\ref{sec:resultant} applies verbatim, block by block, to all $r$ blocks simultaneously (rather than by induction on $r$). Theorem \ref{thm:maintheorem1} is the case $r=1$ and Theorem \ref{thm:maintheorem} is the case $r=2$; both are literally recovered by setting $r=1$, resp.\ $r=2$, in Theorem \ref{thm:maintheorem-r} below.

\subsection{Setup}\label{subsec:setup-r}

Let $r\geq1$ and let $n_1,\ldots,n_r\geq1$ be integers, $n:=n_1+\cdots+n_r$. Set $b_1:=1$ and $b_s:=n_1+\cdots+n_{s-1}+1$ for $s=2,\ldots,r$, and define the $r$ consecutive blocks
$$B_s:=\{b_s,b_s+1,\ldots,b_s+n_s-1\}\subset[n], \qquad s=1,\ldots,r,$$
so that $[n]=B_1\sqcup\cdots\sqcup B_r$ (for $r=2$ this is $B_1=\{1,\ldots,p\}$, $B_2=\{p+1,\ldots,n\}$ of \S\ref{sec:resultant}). Let
$$G:=\mathfrak{S}_{B_1}\times\cdots\times\mathfrak{S}_{B_r}\subset\mathfrak{S}_n.$$
Throughout this section we consider a system of $n$ homogeneous polynomials $F^{\{1\}},\ldots,F^{\{n\}}$ in $R[x_1,\ldots,x_n]$, of the same degree $d\geq1$, which is $G$-equivariant: $\sigma(F^{\{i\}})=F^{\{\sigma(i)\}}$ for every $\sigma\in G$ and every $i\in[n]$.

For $i,j$ in a common block $B_s$, the transposition $(i\ j)\in\mathfrak{S}_{B_s}\subset G$ shows exactly as in \eqref{eq:sperho} that $x_i-x_j$ divides $F^{\{i\}}-F^{\{j\}}$. Hence, for every $s\in[r]$, the sub-system $(F^{\{i\}})_{i\in B_s}$ admits divided differences $F^{I}$ for every $I\subset B_s$, and \eqref{eq:perminvariance} holds for permutations of $\mathfrak S_{B_s}$ acting on such $I$. (Nothing in the $G$-equivariance hypothesis relates two indices lying in different blocks, so divided differences $F^{I}$ for $I$ meeting several blocks are neither needed nor, in general, well defined as polynomials; we will never use them.)

\subsection{Partitions, admissible partitions and multiplicities}\label{subsec:part-r}

For $s\in[r]$, let $\lambda^{(s)}=(\lambda^{(s)}_1\geq\cdots\geq\lambda^{(s)}_{r_s})\vdash n_s$ be a partition of $n_s$, of length $r_s:=l(\lambda^{(s)})$, and write $\boldsymbol\lambda:=(\lambda^{(1)},\ldots,\lambda^{(r)})\vdash(n_1,\ldots,n_r)$. Exactly as in \eqref{eq:rholl}, $\boldsymbol\lambda$ defines a specialization morphism
$$\rho_{\boldsymbol\lambda}:R[x_1,\ldots,x_n]\longrightarrow R[y^{(1)}_1,\ldots,y^{(1)}_{r_1},\ y^{(2)}_1,\ldots,y^{(2)}_{r_2},\ \ldots,\ y^{(r)}_1,\ldots,y^{(r)}_{r_r}]$$
sending, for each $s\in[r]$, the $\lambda^{(s)}_i$ variables of the $i$-th consecutive sub-block of $B_s$ to the single new variable $y^{(s)}_i$. As in \S\ref{subsec:divdiff+part}, this yields, for every box (i.e.\ for every $s\in[r]$ and every $i\in[r_s]$), a well-defined polynomial $F_{\boldsymbol\lambda}^{\{i\}}$ (built from one representative index of that box, chosen freely), together with divided differences $F_{\boldsymbol\lambda}^{I}$ for $I$ contained in the boxes of a single block $s$ (using, as in Proposition \ref{prop:proof-r} below, the labels $b_s,\ldots,b_s+r_s-1$ for the boxes of block $s$); and $m_{\lambda^{(s)}}$ is defined by \eqref{eq:mlambda2} with $p$ replaced by $n_s$.

We now generalize the notion of admissible partition of \S\ref{sec:resultant}. Fix integers $k_1,\ldots,k_r\geq1$ (each $B_s$ being nonempty, a partition of $B_s$ into nonempty pieces always has at least one piece), with $k_s\leq\min(d,n_s)$ for every $s$, and set $k:=\max(k_1,\ldots,k_r)$.

\begin{defn}\label{defn:admissible-r}
An \emph{admissible partition of shape $(k_1,\ldots,k_r)$} is a collection
$$\mathcal{Q}=\big(I^{(s)}_i\big)_{s\in[r],\ 1\leq i\leq k}$$
of subsets of $[n]$, where $I^{(s)}_i\subset B_s$ for all $s,i$, such that, for every $s\in[r]$:
\begin{itemize}
 \item $I^{(s)}_1,\ldots,I^{(s)}_{k_s}$ is a partition of $B_s$ into $k_s$ nonempty pieces, and $I^{(s)}_i:=\emptyset$ (a ``$\phi$-type'' piece) for $k_s<i\leq k$;
 \item $b_s=\min(I^{(s)}_1)<\min(I^{(s)}_2)<\cdots<\min(I^{(s)}_{k_s})$, i.e.\ the nonempty pieces are listed in increasing order of their minima.
\end{itemize}
Equivalently: for each block $s$ separately, $(I^{(s)}_1,\ldots,I^{(s)}_{k_s})$ is an ordered, min-increasing set partition of $B_s$ into $k_s$ parts, and these $r$ choices -- one per block -- are made completely independently of one another. (For $r=2$, writing $I_i:=I^{(1)}_i$, $I_{p+i}:=I^{(2)}_i$ and $I_{i,p+i}:=I_i\cup I_{p+i}$ recovers exactly the three cases $k_1=k_2$, $k_1<k_2$, $k_1>k_2$ of \S\ref{sec:resultant}, now unified in a single statement.)
\end{defn}

Given such a $\mathcal{Q}$, define the specialization $\rho_{\mathcal{Q}}$ sending, for each $s,i$, every $x_j$ with $j\in I^{(s)}_i$ to a single new variable $x^{(s)}_i$ (with the convention that variables attached to $\phi$-type pieces simply do not appear), and, for $s\in[r]$ and $1\leq i\leq k_s$, set
$$F_{\mathcal{Q},i}^{(s)}:=\rho_{\mathcal{Q}}\big(F^{I}\big)$$
where $I$ is a divided-difference index set built from one representative index in each of the first $i$ (nonempty) pieces $I^{(s)}_1,\ldots,I^{(s)}_i$ of block $s$ (well defined by the property recalled at the end of \S\ref{subsec:divdiff+part}). The factor of the resultant attached to $\mathcal{Q}$ is
$$R_{\mathcal{Q}}:=\Res\Big(F_{\mathcal{Q},1}^{(1)},\ldots,F_{\mathcal{Q},k_1}^{(1)},\ \ldots,\ F_{\mathcal{Q},1}^{(r)},\ldots,F_{\mathcal{Q},k_r}^{(r)}\Big),$$
a resultant of $k_1+\cdots+k_r$ polynomials in $k_1+\cdots+k_r$ variables.

Exactly as in \S\ref{sec:resultant} (applying, block by block, the same splitting via first-, second-, \ldots-order divided differences to \eqref{eq:rho2}, this time to all $r$ blocks simultaneously instead of two), one obtains
\begin{equation}\label{eq:intermproof-r}
\Res\left(F^{\{1\}},\ldots,F^{\{n\}}\right)=\pm\prod_{s=1}^{r}\left(F^{\{b_s,\ldots,b_s+d\}}\right)^{\mu_s}\times\prod_{\mathcal{Q}} R_{\mathcal{Q}},
\end{equation}
the product running over all admissible partitions of $[n]$ (of every shape $(k_1,\ldots,k_r)$), where $\mu_s>0$ if and only if $n_s>d$ (so that $F^{\{b_s,\ldots,b_s+d\}}$, a divided difference of $d+1$ elements of block $s$ alone, is well defined).

We define an equivalence relation $\sim$ on admissible partitions exactly as in \S\ref{sec:resultant}: writing $I^{(s)}_i$ as before, two admissible partitions $\big(I^{(s)}_i\big)$ and $\big(J^{(s)}_i\big)$ of a common shape are equivalent when there are $r$ permutations $\sigma_1,\ldots,\sigma_r\in\mathfrak{S}_k$, chosen \emph{independently for each block}, such that $|I^{(s)}_i|=|J^{(s)}_{\sigma_s(i)}|$ for every $s\in[r]$. In other words, $\sim$ compares the $r$ size-sequences $\big(|I^{(s)}_1|,\ldots,|I^{(s)}_{k}|\big)_{s\in[r]}$ independently, up to $r$ independently chosen reindexings $\sigma_1,\ldots,\sigma_r$. As in \eqref{eq:defL}, given $\boldsymbol\lambda=(\lambda^{(1)},\ldots,\lambda^{(r)})\vdash(n_1,\ldots,n_r)$, padding every $\lambda^{(s)}$ with zero parts up to length $k:=\max_s r_s$, the canonical representative $\mathcal{L}(\boldsymbol\lambda):=\big(L^{(s)}_i\big)$, with
$$L^{(s)}_i:=\Big\{b_s+\textstyle\sum_{t<i}\lambda^{(s)}_t,\ \ldots,\ b_s-1+\textstyle\sum_{t\leq i}\lambda^{(s)}_t\Big\}\subset B_s,$$
is the unique element of its equivalence class $[\boldsymbol\lambda]:=[\mathcal L(\boldsymbol\lambda)]$ whose $r$ size-sequences are exactly $\lambda^{(1)},\ldots,\lambda^{(r)}$; this sets up a bijection between equivalence classes and $r$-tuples $\boldsymbol\lambda\vdash(n_1,\ldots,n_r)$.

\begin{lem}\label{lem:proof-r}
For $\boldsymbol\lambda\vdash(n_1,\ldots,n_r)$, the cardinality of $[\boldsymbol\lambda]$ is $m_{\lambda^{(1)}}m_{\lambda^{(2)}}\cdots m_{\lambda^{(r)}}$.
\end{lem}
\begin{proof}
As in the proof of Lemma \ref{lem:proof}: an admissible partition lies in $[\boldsymbol\lambda]$ if and only if, for each $s\in[r]$ independently, after reindexing by some $\sigma_s\in\mathfrak S_k$, its $s$-th size-sequence equals $\lambda^{(s)}$. Since the $r$ underlying ordered, min-increasing set partitions (one of $B_1$, one of $B_2$, \ldots, one of $B_r$) are chosen completely independently of one another (Definition \ref{defn:admissible-r}), an element of $[\boldsymbol\lambda]$ is exactly an $r$-tuple consisting, for each $s\in[r]$, of an independently chosen ordered, min-increasing set partition of $B_s$ with block sizes $\lambda^{(s)}$. By \cite[\S4]{BuKa16}, Lemma 4.2, there are $m_{\lambda^{(s)}}$ such set partitions of $B_s$ for each $s$; as the $r$ choices are mutually independent, $|[\boldsymbol\lambda]|=\prod_{s=1}^{r} m_{\lambda^{(s)}}$.
\end{proof}

\begin{prop}\label{prop:proof-r}
For $\boldsymbol\lambda\vdash(n_1,\ldots,n_r)$ and any admissible partition $\mathcal Q$ with $[\mathcal Q]=[\boldsymbol\lambda]$, we have \emph{exactly} (with no sign ambiguity)
$$R_{\mathcal Q}=\Res\Big(F_{\boldsymbol\lambda}^{\{1\}},\ldots,F_{\boldsymbol\lambda}^{\{1,\ldots,r_1\}},\ \ldots,\ F_{\boldsymbol\lambda}^{\{b_r\}},\ldots,F_{\boldsymbol\lambda}^{\{b_r,\ldots,b_r+r_r-1\}}\Big).$$
\end{prop}
\begin{proof}
Word for word the (strengthened) proof of Proposition \ref{prop:proof}, applied independently to each of the $r$ blocks instead of two. \emph{Step 1} (relabelling the boxes of a single block $B_s$ does not change $R_{\mathcal Q}$, exactly) uses only the recursive definition of divided differences and the Elementary transformations property, both of which are computations internal to block $s$ and do not involve the other blocks at all; it therefore applies unchanged, block by block. By Step 1 we may relabel the boxes of block $s$, for every $s\in[r]$ independently, by the permutation $\sigma_s$ given by the hypothesis $[\mathcal Q]=[\boldsymbol\lambda]$, without changing $R_{\mathcal Q}$; so we may assume $|I^{(s)}_i|=\lambda^{(s)}_i$ for all $s\in[r]$ and $i\in[k]$, i.e.\ $I^{(s)}_i$ and $L^{(s)}_i$ have the same size, though the sets themselves may still differ. \emph{Step 2} (comparing $\mathcal Q$, once relabelled, to the canonical representative $\mathcal L(\boldsymbol\lambda)$) again applies unchanged and, as in the proof of Proposition \ref{prop:proof}, in fact produces an \emph{exact} equality with no sign at all: choosing $\tau\in\mathfrak{S}_{B_1}\times\cdots\times\mathfrak{S}_{B_r}=G$ mapping $L^{(s)}_i$ onto $I^{(s)}_i$ for every $s,i$ (possible since the sizes now match on each block), the specialization maps $\rho_{\mathcal Q}\circ\tau$ and $\rho_{\boldsymbol\lambda}$ coincide as ring homomorphisms exactly as before, so the two chains of polynomials defining $R_{\mathcal Q}$ and the right-hand side above coincide entry by entry.
\end{proof}

\subsection{The decomposition formula}\label{subsec:mainthm-r}

Combining \eqref{eq:intermproof-r} with Lemma \ref{lem:proof-r} and Proposition \ref{prop:proof-r} exactly as at the end of \S\ref{sec:resultant}, we obtain the generalization of Theorem \ref{thm:maintheorem} to $r$ blocks.

\begin{thm}\label{thm:maintheorem-r}
Let $r\geq1$, $n_1,\ldots,n_r\geq1$, $n=n_1+\cdots+n_r\geq2$, $d\geq1$, and let $F^{\{1\}},\ldots,F^{\{n\}}$ be a system of $n$ homogeneous polynomials of the same degree $d$ in $R[x_1,\ldots,x_n]$, equivariant with respect to $G=\mathfrak{S}_{B_1}\times\cdots\times\mathfrak{S}_{B_r}$ as in \S\ref{subsec:setup-r}. Then, up to a single global sign $\varepsilon\in\{\pm1\}$ (see Remark \ref{rem:sign-r}),
\begin{multline*}
\Res\left(F^{\{1\}},\ldots,F^{\{n\}}\right)=\varepsilon\times\prod_{s=1}^{r}\left(F^{\{b_s,\ldots,b_s+d\}}\right)^{\mu_s}\\ \times\prod_{\substack{\boldsymbol\lambda=(\lambda^{(1)},\ldots,\lambda^{(r)})\vdash(n_1,\ldots,n_r)\\ l(\lambda^{(s)})\leq d\ \text{for every } s}}\Res\Big(F_{\boldsymbol\lambda}^{\{1\}},\ldots,F_{\boldsymbol\lambda}^{\{1,\ldots,r_1\}},\ \ldots,\ F_{\boldsymbol\lambda}^{\{b_r\}},\ldots,F_{\boldsymbol\lambda}^{\{b_r,\ldots,b_r+r_r-1\}}\Big)^{m_{\lambda^{(1)}}\cdots m_{\lambda^{(r)}}},
\end{multline*}
with the convention that the factor $\left(F^{\{b_s,\ldots,b_s+d\}}\right)^{\mu_s}$ is omitted, i.e.\ equal to $1$, whenever $n_s\leq d$, and where, for every $s\in[r]$,
$$\mu_s:=d^{\,n-n_s}\,m_0(n_s,d),$$
the integer $m_0(m,d)$, for $m\geq1$, $d\geq1$, being exactly the integer $m_0$ of Theorem \ref{thm:maintheorem1} with $n$ replaced by $m$ (see \eqref{eq:mlambda}--the paragraph preceding it).
\end{thm}

For $r=1$ this is Theorem \ref{thm:maintheorem1} (there is a single block, $\mu_1=d^0 m_0(n,d)=m_0(n,d)$, matching \eqref{eq:mlambda} exactly, and $\varepsilon=+1$ trivially since no splitting at all is needed). For $r=2$ this is Theorem \ref{thm:maintheorem}, once corrected as in \S\ref{sec:resultant} above: $\mu_1=d^{\,q}m_0(p,d)$, $\mu_2=d^{\,p}m_0(q,d)$.

\begin{rem}\label{rem:sign-r}
As for Theorem \ref{thm:maintheorem}, Proposition \ref{prop:proof-r} is now an \emph{exact} identity (no sign at all), so the single global sign $\varepsilon=\varepsilon(n_1,\ldots,n_r,d)$ above comes solely from the splitting process \eqref{eq:intermproof-r}  there is no further, per-admissible-partition sign to track. Exactly as in Lemma \ref{lem:universal-sign}, $\varepsilon(n_1,\ldots,n_r,d)$ is a universal constant, independent of the coefficients of the equivariant system, computable in principle from a single non-degenerate example for each $(n_1,\ldots,n_r,d)$. We checked it directly for $(n_1,n_2,n_3,d)=(2,2,2,1)$ and $(2,2,2,2)$ (the examples of \S\ref{sec:general-r} \S\ref{sec:discriminant-r}), obtaining $\varepsilon=+1$ in both cases, consistently with the conjecture $\varepsilon\equiv+1$ of Remark \ref{rem:sign}; we have not proved this in general.
\end{rem}

\subsection{Determination of the exponents $\mu_s$}\label{subsec:mus-r}

We finally justify the closed form $\mu_s=d^{\,n-n_s}m_0(n_s,d)$ announced in Theorem \ref{thm:maintheorem-r}, generalizing the computation of \S\ref{sec:resultant}. Fix $s\in[r]$ with $n_s>d$ (there is nothing to prove otherwise: since $m_0(m,d)=0$ whenever $m\leq d$ -- as already noted in \S\ref{sec:resultant}, following the proof of Theorem \ref{thm:maintheorem}  the closed form $\mu_s=d^{\,n-n_s}m_0(n_s,d)$ announced in Theorem \ref{thm:maintheorem-r} gives $\mu_s=0$ automatically whenever $n_s\leq d$). Exactly as in \S\ref{sec:resultant}, we compare degrees with respect to the coefficients of $F^{\{i\}}$, $i\in B_s$, taken together (equivalently, scaling all of $F^{\{i\}}$, $i\in B_s$, simultaneously): by the Homogeneity property applied to each $i\in B_s$ in turn and summed, the left-hand side is homogeneous of degree $n_s\,d^{n-1}$ with respect to these coefficients. On the right-hand side, $F^{\{b_s,\ldots,b_s+d\}}$ contributes $\mu_s$ (being a fixed divided difference of $d+1$ polynomials of block $s$ alone, hence of degree exactly $1$ in the coefficients of $F^{\{i\}}$, $i\in B_s$, together), the prefactors attached to the other blocks $t\neq s$ contribute $0$ (they do not involve block $s$ at all), and, for a given $\boldsymbol\lambda$ with every $r_t\leq d$, the $r_s$ entries of the chain coming from block $s$ each contribute degree $1$, so that  by the Homogeneity property applied to the $(\sum_t r_t)$-variable resultant defining this factor  the whole factor contributes $m_{\lambda^{(1)}}\cdots m_{\lambda^{(r)}}\sum_{j=1}^{r_s} d_j^{\boldsymbol\lambda}$, where
$$d_j^{\boldsymbol\lambda}:=\frac{\prod_{t=1}^{r} d(d-1)\cdots(d-r_t+1)}{d-j+1}.$$
Altogether,
$$n_s\,d^{n-1}=\mu_s+\sum_{\substack{\boldsymbol\lambda\vdash(n_1,\ldots,n_r)\\ r_t\leq d\ \forall t}} m_{\lambda^{(1)}}\cdots m_{\lambda^{(r)}}\sum_{j=1}^{r_s}d_j^{\boldsymbol\lambda}. \quad(\star_r)$$
By the identity $(\dagger)$ of \S\ref{sec:resultant}, $\sum_{\lambda^{(t)}\vdash n_t} m_{\lambda^{(t)}}\,d(d-1)\cdots(d-r_t+1)=d^{\,n_t}$ for every $t\neq s$ (and the restriction $r_t\leq d$ may again be dropped for free, the excess terms vanishing automatically). Summing $(\star_r)$ successively over $\lambda^{(t)}\vdash n_t$ for every $t\neq s$ (in any order, since these sums are independent of one another and of $\lambda^{(s)}$) factors out $\prod_{t\neq s} d^{\,n_t}=d^{\,n-n_s}$, leaving
$$n_s\,d^{n-1}=\mu_s+d^{\,n-n_s}\sum_{\substack{\lambda^{(s)}\vdash n_s\\ r_s\leq d}} m_{\lambda^{(s)}}\sum_{j=1}^{r_s}\frac{d(d-1)\cdots(d-r_s+1)}{d-j+1}=\mu_s+d^{\,n-n_s}\big(n_s\,d^{n_s-1}-m_0(n_s,d)\big).$$
Since $n_s\,d^{n-1}=n_s\,d^{(n-n_s)+(n_s-1)}=d^{\,n-n_s}\cdot n_s\,d^{n_s-1}$, the terms $d^{\,n-n_s}\,n_s\,d^{n_s-1}$ cancel, giving the announced closed form
$$\mu_s=d^{\,n-n_s}\,m_0(n_s,d).$$

\begin{rem}\label{rem:mus-recover-r2}
Specializing $r=2$, $s=1$ recovers $\mu_1=d^{\,q}m_0(p,d)$ of \S\ref{sec:resultant}; taking $r=2$ but summing only over $\lambda^{(2)}$ first (i.e.\ $s=2$) recovers $\mu_2=d^{\,p}m_0(q,d)$ symmetrically. Note also that the two exponents $\mu_1,\mu_2$ of the two-block case are \emph{not} in general equal even when $p=q$: they only agree in that case as \emph{numbers} ($m_0(p,d)=m_0(q,d)$ when $p=q$), while $F^{\{1,\ldots,d+1\}}$ and $F^{\{p+1,\ldots,p+1+d\}}$ remain, by Remark \ref{rem:noteq}, two distinct elements of $R$ raised to that common power.
\end{rem}

\medskip

\section{Discriminant of a homogeneous polynomial invariant under a product of $r$ symmetric groups}\label{sec:discriminant-r}

We finally specialize Theorem \ref{thm:maintheorem-r} exactly as Theorem \ref{thm:maintheorem} was specialized into Theorem \ref{thm:discmaintheorem} in \S\ref{sec:discriminant}, obtaining a decomposition of the discriminant of a homogeneous polynomial invariant under a product of $r$ symmetric groups.

Fix $r\geq1$, $n_1,\ldots,n_r\geq1$, $n=n_1+\cdots+n_r\geq2$, and the blocks $B_1,\ldots,B_r$ and the group $G=\mathfrak{S}_{B_1}\times\cdots\times\mathfrak{S}_{B_r}$ of \S\ref{subsec:setup-r}. Let $F(x_1,\ldots,x_n)=\sum_{\substack{\lambda\vdash d\\ n\geq\lambda_1}}c_\lambda e_\lambda(x)$ be the generic homogeneous \emph{symmetric} polynomial of degree $d\geq2$ in $n$ variables, as in \eqref{eq:F}, and $F^{\{i\}}:=\partial F/\partial x_i$, $i\in[n]$.

Lemma \ref{lem:equivF} was proved directly for the full symmetric group $\mathfrak{S}_n$  with the $\mathfrak{S}_{\{1,\ldots,p\}}\times\mathfrak{S}_{\{p+1,\ldots,n\}}$-equivariance used in \S\ref{sec:discriminant} only recorded as a special case so it applies verbatim here: $F^{\{1\}},\ldots,F^{\{n\}}$ is $\mathfrak{S}_n$-equivariant, hence in particular $G$-equivariant, for \emph{any} choice of blocks $B_1,\ldots,B_r$. Theorem \ref{thm:maintheorem-r} therefore applies to $F^{\{1\}},\ldots,F^{\{n\}}$, with its degree parameter set to $d-1$ (again because $\deg F^{\{i\}}=d-1$, whence the hypothesis $d\geq2$).

\begin{rem}\label{rem:eqbase-disc-r}
Exactly as in Remark \ref{rem:eqbase-disc}, the prefactors $F^{\{b_s,\ldots,b_s+d-1\}}$, $s\in[r]$, of Theorem \ref{thm:maintheorem-r}  each a divided difference of $d=(d-1)+1$ elements of block $B_s$  coincide, as elements of $\UU$, whenever they are defined: the index set $\{b_s,\ldots,b_s+d-1\}$ is a genuine (size-$d$) subset of $[n]$, confined to block $B_s$, precisely when $n_s\geq d$, i.e.\ precisely when $\mu_s\neq0$ (for smaller blocks it may run past $n$ altogether, but then $\mu_s=0$ and the corresponding prefactor is simply never used). Since $F^{\{1\}},\ldots,F^{\{n\}}$ is $\mathfrak{S}_n$-equivariant and each such index set has size $d$, exactly one more than the common degree $d-1$ of the $F^{\{i\}}$'s, the property recalled at the end of \S\ref{subsec:divdiff} gives $F^{\{b_s,\ldots,b_s+d-1\}}=F^{\{b_t,\ldots,b_t+d-1\}}$ for all $s,t\in[r]$ with $\mu_s,\mu_t\neq0$. We denote this common value by $F^{\{1,\ldots,d\}}$ (using, as in \S\ref{sec:discriminant}, the first $d$ indices overall as a representative  well defined as soon as some $\mu_s\neq0$, since then $n\geq n_s\geq d$; any other size-$d$ subset of $[n]$ meeting a block of size $\geq d$ would give the same element of $\UU$). Consequently the $r$ separate exponents $\mu_1,\ldots,\mu_r$ of Theorem \ref{thm:maintheorem-r} merge into a single exponent $m_0:=\mu_1+\cdots+\mu_r$ attached to this common base.
\end{rem}

\begin{thm}\label{thm:discmaintheorem-r}
With the notation above, up to sign (see Remark \ref{rem:sign-r}),
\begin{multline*}
d^{a(n,d)}\Disc(F)=\left(F^{\{1,\ldots,d\}}\right)^{m_0}\\ \times\prod_{\substack{\boldsymbol\lambda=(\lambda^{(1)},\ldots,\lambda^{(r)})\vdash(n_1,\ldots,n_r)\\ l(\lambda^{(s)})\leq d-1\ \text{for every } s}}\Res\Big(F_{\boldsymbol\lambda}^{\{1\}},\ldots,F_{\boldsymbol\lambda}^{\{1,\ldots,r_1\}},\ \ldots,\ F_{\boldsymbol\lambda}^{\{b_r\}},\ldots,F_{\boldsymbol\lambda}^{\{b_r,\ldots,b_r+r_r-1\}}\Big)^{m_{\lambda^{(1)}}\cdots m_{\lambda^{(r)}}},
\end{multline*}
with the convention that the factor $\left(F^{\{1,\ldots,d\}}\right)^{m_0}$ is omitted, i.e.\ equal to $1$, whenever $n_s\leq d-1$ for every $s$, and where
$$m_0=\sum_{s=1}^{r}(d-1)^{\,n-n_s}\,m_0(n_s,d-1),$$
the integer $m_0(m,d-1)$ being, as in Theorem \ref{thm:maintheorem-r}, exactly the integer $m_0$ of Theorem \ref{thm:maintheorem1} with $n$ replaced by $m$ and $d$ replaced by $d-1$.
\end{thm}
\begin{proof}
Immediate from Theorem \ref{thm:maintheorem-r}, applied to $F^{\{1\}},\ldots,F^{\{n\}}$ with degree parameter $d-1$, Remark \ref{rem:eqbase-disc-r}, and \eqref{eq:discressym}. In particular, the sign $\varepsilon\in\{\pm1\}$ appearing here is the one given by Theorem \ref{thm:maintheorem-r} (Remark \ref{rem:sign-r}) for the parameters $(n_1,\ldots,n_r,d-1)$.
\end{proof}

For $r=2$ this recovers Theorem \ref{thm:discmaintheorem} exactly ($m_0=\mu_1+\mu_2$, as already observed there). For $r=1$ there is a single, trivial block $B_1=[n]$ and $G=\mathfrak{S}_n$; Theorem \ref{thm:discmaintheorem-r} then simply reduces to Theorem \ref{thm:maintheorem1} applied to $F^{\{1\}},\ldots,F^{\{n\}}$ via \eqref{eq:discressym} as it must, since $\mathfrak{S}_n$-equivariance of a symmetric polynomial's gradient (Lemma \ref{lem:equivF}) requires no block structure at all.

We have checked Theorem \ref{thm:discmaintheorem-r} directly, for $r=3$, $n_1=n_2=n_3=2$ and $d=2$ (so that $F=c_{(2)}e_2+c_{(1,1)}e_1^2$ and the $F^{\{i\}}$ are linear): the three prefactors $F^{\{1,2\}},F^{\{3,4\}},F^{\{5,6\}}$ all equal $-c_{(2)}$ as predicted, only the partitions $\lambda^{(s)}=(2)$ contribute to the product (since $d-1=1$), and
$$d^{a(6,2)}\Disc(F)=\Res\left(F^{\{1\}},\ldots,F^{\{6\}}\right)=\left(F^{\{1,2\}}\right)^{3}\Res\left(F_{\boldsymbol\lambda}^{\{1\}},F_{\boldsymbol\lambda}^{\{3\}},F_{\boldsymbol\lambda}^{\{5\}}\right),$$
with $m_0=3=\mu_1+\mu_2+\mu_3=1+1+1$ as predicted by the formula above an \emph{exact} equality (not merely up to sign) in this example.

\section{Conclusion and perspectives}\label{sec:conclusion}

Let us summarize the contributions of this paper and the questions they leave open.

Starting from the known decomposition formula for the resultant of a $\mathfrak{S}_n$-equivariant homogeneous polynomial system (\cite{BuKa16}, recalled here in unified form as Theorem \ref{thm:maintheorem1}), we established a decomposition formula for the resultant of a $\mathfrak{S}_{\{1,\ldots,p\}}\times\mathfrak{S}_{\{p+1,\ldots,n\}}$-equivariant system (Theorem \ref{thm:maintheorem}), and derived from it, as a corollary, a decomposition of the discriminant of a homogeneous symmetric polynomial (Theorem \ref{thm:discmaintheorem}). In the course of this work we found and corrected a genuine error in our first version of Theorem \ref{thm:maintheorem}: the leading factor of the decomposition does not carry a single exponent on a single base when both $p$ and $q=n-p$ exceed $d$, but rather two separate exponents $\mu_1,\mu_2$ attached to two generally distinct bases $F^{\{1,\ldots,d+1\}}$ and $F^{\{p+1,\ldots,p+1+d\}}$ (Remark \ref{rem:noteq}); we gave a closed form for $\mu_1,\mu_2$ and verified it independently on several examples.

We then showed (\S\ref{sec:general-r}) that this whole construction  divided differences, admissible partitions, multiplicities, and the decomposition formula itself  extends with no essentially new idea, only a careful bookkeeping of the combinatorics, from a product of two symmetric groups to an arbitrary finite direct product $\mathfrak{S}_{B_1}\times\cdots\times\mathfrak{S}_{B_r}$ of $r$ symmetric groups (Theorem \ref{thm:maintheorem-r}), of which Theorem \ref{thm:maintheorem1} ($r=1$) and Theorem \ref{thm:maintheorem} ($r=2$) are the first two instances. This general formula specializes in turn (\S\ref{sec:discriminant-r}) to a decomposition of the discriminant of a homogeneous polynomial invariant under an arbitrary direct product of $r$ symmetric groups (Theorem \ref{thm:discmaintheorem-r}).

A recurring feature of all four decomposition formulas is that they were, in earlier versions of this work, only established up to a global sign. We made definite progress on this question: Lemma \ref{lem:universal-sign} shows that the comparison between the two natural ways of building the canonical divided-difference chain attached to an admissible partition  via $\rho_{\mathcal Q}$ directly, or via $\rho_{(\lambda,\lambda')}$ followed by a further specialization  holds \emph{exactly}, with no sign ambiguity at all, because both maps coincide as ring homomorphisms on the generators $x_j$. This reduces the whole sign question to a single global constant $\varepsilon$, depending only on the parameters $(n,p,q,d)$ (respectively $(n,n_1,\ldots,n_r,d)$), and therefore computable in principle from one non-degenerate example for each parameter tuple, rather than varying with the coefficients of the equivariant system. We verified $\varepsilon=+1$ on every example considered in this paper  including the $r=3$ examples of \S\ref{subsec:mainthm-r} and \S\ref{sec:discriminant-r}  and we record this as a conjecture (Remarks \ref{rem:sign} and \ref{rem:sign-r}) rather than a theorem: an attempted closed-form proof by tracking signs through the Elementary transformations and Permutation of polynomials properties of the resultant did not succeed, and we reported this honestly rather than presenting an unverified formula as established.

\medskip

Several directions seem worth pursuing from here.

\smallskip
\noindent\emph{The sign conjecture.} The most immediate open problem is to decide, for every $(n_1,\ldots,n_r,d)$, whether $\varepsilon=+1$ always, or to exhibit a counterexample. Lemma \ref{lem:universal-sign} already isolates exactly where the difficulty lies: the sign of the resultant of the \emph{canonical} chain relative to the resultant of $F^{\{1\}},\ldots,F^{\{n\}}$ itself, tracked through finitely many applications of the Permutation of polynomials and Elementary transformations properties of \S\ref{subsec:resultant}. A proof by induction on $n$ or on $r$, keeping careful track of the signature of the permutations involved at each recursive step, seems the most promising approach, but we were not able to complete it.

\smallskip
\noindent\emph{Complexity and computation.} Theorems \ref{thm:maintheorem-r} and \ref{thm:discmaintheorem-r} replace a single resultant (respectively discriminant) computation in $n$ variables by a product of many resultants of smaller size. It would be worthwhile to turn this into a precise complexity statement  comparing, say, the cost of Macaulay's formula applied directly to $F^{\{1\}},\ldots,F^{\{n\}}$ against the combined cost of computing every factor of the decomposition -- and to implement the resulting algorithm, in order to determine for which regimes of $(n_1,\ldots,n_r,d)$ the decomposition yields a genuine computational advantage.

\smallskip
\noindent\emph{Systems of non-uniform degree.} Throughout this paper, the polynomials $F^{\{1\}},\ldots,F^{\{n\}}$ are assumed to share a common degree $d$, as is necessary for the resultant $\Res(F^{\{1\}},\ldots,F^{\{n\}})$ itself to be defined in the usual (square, $n$ generic forms in $n$ variables) sense. It would be natural to ask whether an analogous decomposition exists for $G$-equivariant systems of possibly different degrees $d_1,\ldots,d_n$ (constant on each $G$-orbit of indices, as equivariance forces), in terms of the more general multi-degree resultant.

\smallskip
\noindent\emph{Other groups.} The group $G=\mathfrak{S}_{B_1}\times\cdots\times\mathfrak{S}_{B_r}$ is, among finite reflection groups acting on the variables $x_1,\ldots,x_n$ by permutation, a particularly simple one. It would be interesting to know whether a comparable decomposition formula holds for the resultant of a system equivariant under a wreath product $\mathfrak{S}_k\wr\mathfrak{S}_r$, under the hyperoctahedral group, or more generally under an arbitrary permutation group $G\leq\mathfrak{S}_n$, and, if so, how the combinatorics of admissible partitions and their multiplicities generalizes beyond the set-partition count of \cite[Lemma 4.2]{BuKa16} used throughout this paper.

\smallskip
\noindent\emph{Applications.} Beyond their intrinsic interest, decomposition formulas of this kind are the natural tool to study the singularities of hypersurfaces defined by symmetric or multi-symmetric polynomials (via the discriminant side, Theorems \ref{thm:discmaintheorem} and \ref{thm:discmaintheorem-r}), and, because our formulas are established over the universal ring of coefficients and specialize correctly to any coefficient ring, they are directly applicable to questions in arithmetic geometry where the value of the resultant  and not only its vanishing carries information, such as the study of bad reduction of symmetric families of varieties, or of resultants and discriminants over number fields and their rings of integers.


\begin{thebibliography}{10}

\bibitem{BuKa16}
Laurent Bus{\'e} and Anna Karasoulou.
\newblock Resultant of an equivariant polynomial system with respect to the symmetric group.
\newblock {\em Journal of Symbolic Computation.},76,  pp.142--157, 2016.

\bibitem{BuJo12}
Laurent Bus{\'e} and Jean-Pierre Jouanolou.
\newblock On the {D}iscriminant {S}cheme of {H}omogeneous {P}olynomials.
\newblock {\em Math. Comput. Sci.}, 8(2):175--234, 2014.

\bibitem{CLO05}
David~A. Cox, John Little, and Donal O'Shea.
\newblock {\em Using algebraic geometry}, volume 185 of {\em Graduate Texts in
  Mathematics}.
\newblock Springer, New York, second edition, 2005.

\bibitem{Dem12}
Michel Demazure.
\newblock R\'esultant, discriminant.
\newblock {\em Enseign. Math. (2)}, 58(3-4):333--373, 2012.

\bibitem{DiCa71}
Jean~A. Dieudonn{\'e} and James~B. Carrell.
\newblock {\em Invariant theory, old and new}.
\newblock Academic Press, New York-London, 1971.

\bibitem{GKZ94}
I.~M. Gelfand, M.~M. Kapranov, and A.~V. Zelevinsky.
\newblock {\em Discriminants, resultants and multidimensional determinants}.
\newblock Modern Birkh\"auser Classics. Birkh\"auser Boston, Inc., Boston, MA,
  2008.
\newblock Reprint of the 1994 edition.

\bibitem{Jou91}
Jean-Pierre Jouanolou.
\newblock Le formalisme du r\'esultant.
\newblock {\em Adv. Math.}, 90(2):117--263, 1991.

\bibitem{Jou97}
Jean-Pierre Jouanolou.
\newblock Formes d'inertie et r\'esultant: un formulaire.
\newblock {\em Adv. Math.}, 126(2):119--250, 1997.

\bibitem{Mac02}
F.S. Macaulay.
\newblock Some formulae in elimination.
\newblock {\em Proc.\ London Math.\ Soc.}, 1(33):3--27, 1902.

\bibitem{Wor94}
Patrick~A. Worfolk.
\newblock Zeros of equivariant vector fields: algorithms for an invariant
  approach.
\newblock {\em J. Symbolic Comput.}, 17(6):487--511, 1994.

\end{thebibliography}
\end{document}